\documentclass[notitlepage,11pt,reqno]{amsart}
\usepackage[foot]{amsaddr}
\RequirePackage[OT1]{fontenc}
\RequirePackage[normalem]{ulem}
\RequirePackage{amssymb,nicefrac,bm,upgreek,mathtools,verbatim,enumerate}
\RequirePackage[mathscr]{eucal}
\RequirePackage{dsfont}
\RequirePackage[final]{hyperref}

\usepackage[margin=1in]{geometry}
\allowdisplaybreaks
\newcommand{\stkout}[1]{\ifmmode\text{\sout{\ensuremath{#1}}}\else\sout{#1}\fi}
\newcommand{\ttup}[1]{\textup{(}#1\textup{)}}
\usepackage{color}
\definecolor{dmagenta}{rgb}{.6,.0,.8}
\definecolor{dblue}{rgb}{.0,.0,.5}
\definecolor{mblue}{rgb}{.0,.0,.8}
\definecolor{ddblue}{rgb}{.0,.0,.4}
\definecolor{dred}{rgb}{.6,.0,.0}
\definecolor{dgreen}{rgb}{.0,.5,.0}
\definecolor{Eeom}{rgb}{.0,.0,.5}

\newtheorem{lemma}{Lemma}[section]
\newtheorem{theorem}{Theorem}[section]

\newtheorem{corollary}{Corollary}[section]

\theoremstyle{definition}
\newtheorem{definition}{Definition}[section]
\newtheorem{assumption}{Assumption}[section]

\newtheorem{example}{Example}[section]

\theoremstyle{remark}
\newtheorem{remark}{Remark}[section]
\numberwithin{equation}{section}

\newcommand{\compcent}[1]{\vcenter{\hbox{$#1\circ$}}}

\newcommand{\comp}{\mathbin{\mathchoice
{\compcent\scriptstyle}{\compcent\scriptstyle}
{\compcent\scriptscriptstyle}{\compcent\scriptscriptstyle}}}

\hypersetup{
  colorlinks=true,
  citecolor=mblue,
  linkcolor=mblue,
  frenchlinks=false,
  pdfborder={0 0 0},
  naturalnames=false,
  hypertexnames=false,
  breaklinks}
\usepackage[capitalize,nameinlink]{cleveref}
\usepackage[abbrev,msc-links,nobysame,citation-order]{amsrefs} 

\crefname{section}{Section}{Sections}
\crefname{subsection}{Section}{Sections}
\crefname{hypothesis}{Hypothesis}{Conditions}
\crefname{assumption}{Assumption}{Assumptions}
\crefname{lemma}{Lemma}{Lemmas}

\Crefname{figure}{Figure}{Figures}

\crefformat{equation}{\textup{#2(#1)#3}}
\crefrangeformat{equation}{\textup{#3(#1)#4--#5(#2)#6}}
\crefmultiformat{equation}{\textup{#2(#1)#3}}{ and \textup{#2(#1)#3}}
{, \textup{#2(#1)#3}}{, and \textup{#2(#1)#3}}
\crefrangemultiformat{equation}{\textup{#3(#1)#4--#5(#2)#6}}%
{ and \textup{#3(#1)#4--#5(#2)#6}}{, \textup{#3(#1)#4--#5(#2)#6}}%
{, and \textup{#3(#1)#4--#5(#2)#6}}

\Crefformat{equation}{#2Equation~\textup{(#1)}#3}
\Crefrangeformat{equation}{Equations~\textup{#3(#1)#4--#5(#2)#6}}
\Crefmultiformat{equation}{Equations~\textup{#2(#1)#3}}{ and \textup{#2(#1)#3}}
{, \textup{#2(#1)#3}}{, and \textup{#2(#1)#3}}
\Crefrangemultiformat{equation}{Equations~\textup{#3(#1)#4--#5(#2)#6}}%
{ and \textup{#3(#1)#4--#5(#2)#6}}{, \textup{#3(#1)#4--#5(#2)#6}}%
{, and \textup{#3(#1)#4--#5(#2)#6}}

\crefdefaultlabelformat{#2\textup{#1}#3}
\newcommand{\process}[1]{{\{#1_t\}_{t\ge0}}}

\newcommand{\sF}{\mathfrak{F}}
\newcommand{\Bor}{{\mathfrak{B}}} 
\newcommand{\cP}{{\mathcal{P}}}   
\newcommand{\eom}{{\mathscr{G}}}  
\newcommand{\sC}{{\mathscr{C}}}   

\newcommand{\Ind}{\mathds{1}} 
\newcommand{\Cc}{\mathcal{C}} 

\newcommand{\Lp}{{L}} 
\newcommand{\Lpl}{{L}_{\text{loc}}}   
\newcommand{\Sob}{{\mathscr W}}
\newcommand{\Sobl}{{\mathscr W}_{\text{loc}}} 

\newcommand{\sA}{{\mathscr{A}}} 
\newcommand{\Ag}{{\mathcal{A}}}  
\newcommand{\Lg}{{\mathcal{L}}}  
\newcommand{\rc}{{\mathscr{R}}}   
\newcommand{\Act}{{\mathbb{U}}}   
\newcommand{\Uadm}{{\mathfrak{U}}}
\newcommand{\Usm}{{{\mathfrak{U}}_{\mathrm{sm}}}} 
\newcommand{\Ussm}{{{\mathfrak{U}}_{\mathrm{ssm}}}}   
\newcommand{\Ussmb}{{{\Tilde{\mathfrak{U}}}_{\mathrm{ssm}}}} 

\newcommand{\RR}{\mathds{R}}  
\newcommand{\NN}{\mathds{N}}  
\newcommand{\DD}{\mathds{D}}
\newcommand{\Rd}{{\mathds{R}^{d}}}
\DeclareMathOperator{\Exp}{\mathbb{E}}
\DeclareMathOperator{\Prob}{\mathbb{P}}
\newcommand{\D}{\mathrm{d}}
\newcommand{\E}{\mathrm{e}}

\newcommand{\sB}{{\mathscr{B}}}  

\newcommand{\cI}{{\mathcal{I}}}  
\newcommand{\sI}{{\mathscr{I}}}  
\newcommand{\sJ}{{\mathscr{J}}}  

\newcommand{\cK}{{\mathcal{K}}}  
\newcommand{\Vo}{\mathscr{V}_{\!\circ}}  
\newcommand{\sV}{\mathscr{V}}  
\newcommand{\cN}{{\mathcal{N}}} 

\newcommand{\abs}[1]{\lvert#1\rvert}
\newcommand{\norm}[1]{\lVert#1\rVert}
\newcommand{\babs}[1]{\bigl\lvert#1\bigr\rvert}

\newcommand{\babss}[1]{\biggl\lvert#1\biggr\rvert}
\newcommand{\bnorm}[1]{\bigl\lVert#1\bigr\rVert}

\newcommand{\bnormm}[1]{\biggl\lVert#1\biggr\rVert}
\newcommand{\transp}{^{\mathsf{T}}}
\newcommand{\df}{:=}
\DeclareMathOperator*{\osc}{osc}
\DeclareMathOperator*{\Argmin}{Arg\,min}
\DeclareMathOperator*{\diag}{diag}

\newcommand{\order}{{\mathscr{O}}}
\newcommand{\sorder}{{\mathfrak{o}}}
\newcommand{\grad}{\nabla}
\newcommand{\uuptau}{{\Breve\uptau}}

\newcommand{\ttl}{\Large
Ergodic control of diffusions with compound Poisson jumps \\[3pt]
under a general structural hypothesis}
\begin{document}

\title[Ergodic control of diffusions with compound Poisson jumps]{\ttl}
\author[Ari Arapostathis]{Ari Arapostathis$^\dag$}
\address{$^\dag$ Department of Electrical and Computer Engineering\\
The University of Texas at Austin\\
2501 Speedway, EERC 7.824\\
Austin, TX~~78712, USA}
\email{ari@ece.utexas.edu}

\author[Guodong Pang]{Guodong Pang$^\ddag$}
\author[Yi Zheng]{Yi Zheng$^\ddag$}
\address{$^\ddag$ The Harold and Inge Marcus Department of Industrial and
Manufacturing Engineering,
College of Engineering,
Pennsylvania State University,
University Park, PA 16802}
\email{$\lbrace$gup3,yxz282$\rbrace$@psu.edu}

\begin{abstract} 
We study the ergodic control problem for a class of controlled jump diffusions
driven by a compound Poisson process.
This extends the results of
[\textit{SIAM J.\ Control Optim.\/} \textbf{57} (2019), no.~2, 1516--1540]
to running costs that are not near-monotone.
This generality is needed in applications such as optimal scheduling
of large-scale parallel server networks.

We provide a full characterizations of optimality via
the Hamilton--Jacobi--Bellman (HJB) equation, for which we additionally
exhibit regularity of solutions under mild hypotheses.
In addition, we show that optimal stationary Markov controls
are a.s.\ pathwise optimal.
Lastly, we show that one can fix a stable control outside a compact set
and obtain near-optimal solutions by solving the HJB on a
sufficiently large bounded domain. 
This is useful for constructing asymptotically optimal scheduling policies
for multiclass parallel server networks.
\end{abstract}

\subjclass[2000]{Primary: 93E20, 60J75, 35Q93. Secondary: 60J60, 35F21, 93E15.}

\keywords{controlled jump diffusions, compound Poisson process, ergodic control, Hamilton--Jacobi--Bellman (HJB) equation, stable Markov optimal control, pathwise optimality, approximate HJB equation, 
spatial truncation}

\maketitle

\section{Introduction}

Control problems for jump diffusions have been studied extensively. 
We refer the readers to \cite{BenLi-84} and references therein for the study of the 
discounted problem and many applications.
In \cite{Menaldi-99}, the ergodic control problem under a strong blanket
stability condition  (see \cite[(1.6)]{Menaldi-99}) has been studied.
In \cite{ACPZ19}, the authors have studied the ergodic control problem
for jump diffusions when the associated L\'evy measures
are finite and state-dependent and have rough kernels
under a near-monotone running cost function.
However, in many applications
the dynamics are not stable under any Markov control,
nor do they have a near-monotone running cost function.
In this paper we waive these assumptions, and study
the ergodic control problem
under the more general structural hypotheses
(see \cref{A2.1,A2.2}) first introduced in
\cite{ABP15}, and also used in \cite{AP16a} in the study of multiclass
multi-pool queueing networks.

The class of jump diffusions studied in this paper is abstracted from the 
diffusion limit of multiclass queueing networks in the Halfin--Whitt regime
with service interruptions \cite{APS19}.
The jump process in this model is compound Poisson, and thus
the associated L\'evy measure is finite.
However, it not have any particular regularity properties
such as density.
In addition, the running cost function, which typically penalizes the queue size,
is not near-monotone.
We abstract and generalize this model, and consider a large
class of diffusions with jumps, which includes models having a
near-monotone running cost function, or with uniformly stable dynamics as special cases. 

We first establish the existence of an optimal stationary Markov control
for the ergodic control problem, 
and characterize all optimal stationary Markov controls via the ergodic
Hamilton--Jacobi--Bellman (HJB) equation. 

It is shown in \cite[Example 1.1]{ACPZ19} that the 
Harnack property may fail for
infinitesimal generators of jump diffusions with compound Poisson jumps.
Thus the approach developed in \cite{ABG12,ABP15} for the study of the ergodic HJB equation 
associated with continuous diffusions cannot be applied here.
On the other hand,  the running cost function is assumed near-monotone in \cite{ACPZ19}, 
and thus the infimum of the value function for the discounted problem is attained
in a compact set (see \cite[Theorem~3.2]{ACPZ19}),
and the solutions of the ergodic HJB equation are bounded from below.
In the present paper, we extend the technique developed in \cite{ACPZ19},
and derive the ergodic HJB under \cref{A2.1,A2.2}.
This is rather delicate, and requires an estimate of the negative part
of the solutions of the HJB.

Another difficulty concerns the regularity of solutions of the discounted and 
ergodic HJB equations associated with jump diffusions, when the L\'evy kernel
is rough.
In \cite{ACPZ19}, we show that the solutions have
locally H\"older continuous second order derivatives when the L\'evy measure has a
compact support (see \cite[Remark~3.4]{ACPZ19}).  
In this paper, we present a gradient estimate for solutions of
a class of second order nonlocal equations in \cref{L5.3} using scaling, and employ
this to establish $\Cc^{2,\alpha}$ regularity of the solutions
of the HJB equations in \cref{T5.3}.
  
We also study pathwise optimality of optimal controls for the ergodic control problem.
For continuous diffusion processes, pathwise optimality has
been studied in \cite{PMB00,PWM04,ABG12,AA17,Ghosh-97}.
Pathwise optimality for jump diffusions with near-monotone running cost
is studied in \cite[Theorem~4.4]{ACPZ19}. 
We extend the technique in \cite{AA17},
using also the result on convergence of random
empirical measures for jump diffusions in \cite[Lemma~4.3]{ACPZ19} while providing 
a crucial estimate on the nonlocal term, 
to establish pathwise optimality for the model studied in this paper. 

The ability to synthesize a near-optimal Markov control, by fixing
a suitable stable control outside a large ball and solving the HJB equation
inside the ball plays a crucial role in the study of asymptotic optimality
for multiclass parallel server networks.
This is used in \cite{ABP15,AP18,AP19,ADPZ19} to
construct asymptotically near-optimal scheduling policies for the prelimit system.
In addressing this problem for jump diffusions, 
we first derive a lower bound for supersolutions of a general class of
integro-differential equations in \cref{L7.1}, and then use this
to establish the required result in \cref{T7main,C7.1}. 
In turn, this result is used to establish the asymptotic optimality of
multiclass networks with service interruptions in \cite{APZ19b}.

\subsection{Organization of the paper}
In the next subsection, we summarize the notation used in this paper.  
In \cref{S2}, we introduce the model and state the assumptions.
\cref{S3} contains some examples from queueing networks
whose limiting controlled jump diffusions satisfy these assumptions.
\cref{S4} concerns the existence of optimal stationary Markov controls.
\cref{S5} is devoted to the study of the HJB equations
on the discounted and ergodic control problems.
In \cref{S6}, we study the pathwise optimality for the ergodic control problem.
The characterization of near-optimal controls is studied in \cref{S7}.

\subsection{Notation}
The standard Euclidean norm in $\RR^{d}$ is denoted by $\abs{\,\cdot\,}$,
$\langle\,\cdot\,,\cdot\,\rangle$ denotes the inner product,
and $x\transp$ denotes the transpose of $x\in\Rd$.
The set of nonnegative real numbers is denoted by $\RR_{+}$,
$\NN$ stands for the set of natural numbers, and $\Ind$ denotes
the indicator function.
The minimum (maximum) of two real numbers $a$ and $b$ is denoted by $a\wedge b$ 
($a\vee b$), respectively,
and $a^\pm \df (\pm a)\vee 0$.
The closure, boundary, and the complement
of a set $A\subset\Rd$ are denoted
by $\Bar{A}$, $\partial{A}$, and $A^{c}$, respectively.
We also let $e\df (1,\dotsc,1)\transp$.
For any function $f\colon\RR^{d}\to\RR$
and domain $D\subset\RR$ we define the oscillation of $f$ on $D$ as follows:
\begin{equation*}
\osc_D\,f\,\df\,\sup\,\bigl\{f(x)-f(y)\,\colon\,x,\,y\in D\bigr\}\,.
\end{equation*}

We denote by $\uptau(A)$ the \emph{first exit time} of the process
$\{X_{t}\}$ from the set $A\subset\RR^{d}$, defined by
\begin{equation*}
\uptau(A) \,\df\, \inf\,\{t>0\,\colon\, X_{t}\not\in A\}\,.
\end{equation*}
The open ball of radius $r$ in $\RR^{d}$, centered at $x\in\RR^d$
is denoted by $B_{r}(x)$. 
We write $B_r$ for $B_r(0)$, and let $\uptau_{r}\df \uptau(B_{r})$,
and $\uuptau_{r}\df \uptau(B^{c}_{r})$.

The term \emph{domain} in $\RR^{d}$
refers to a nonempty, connected open subset of the Euclidean space $\RR^{d}$. 
For a domain $D\subset\RR^{d}$,
the space $\Cc^{k}(D)$ ($\Cc^{\infty}(D)$), $k\ge 0$,
refers to the class of all real-valued functions on $D$ whose partial
derivatives up to order $k$ (of any order) exist and are continuous.
By $\Cc^{k,\alpha}(\RR^{d})$ we denote the set of functions that are
$k$-times continuously differentiable and whose $k$-th derivatives are locally
H\"{o}lder continuous with exponent $\alpha$.
The space $\Lp^{p}(D)$, $p\in[1,\infty)$, stands for the Banach space
of (equivalence classes of) measurable functions $f$ satisfying
$\int_{D} \abs{f(x)}^{p}\,\D{x}<\infty$, and $\Lp^{\infty}(D)$ is the
Banach space of functions that are essentially bounded in $D$.
The standard Sobolev space of functions on $D$ whose generalized
derivatives up to order $k$ are in $\Lp^{p}(D)$, equipped with its natural
norm, is denoted by $\Sob^{k,p}(D)$, $k\ge0$, $p\ge1$.
In general, if $\mathcal{X}$ is a space of real-valued functions on $Q$,
$\mathcal{X}_{\mathrm{loc}}$ consists of all functions $f$ such that
$f\varphi\in\mathcal{X}$ for every $\varphi\in\Cc_{\mathrm{c}}^{\infty}(Q)$.
In this manner we obtain for example the space $\Sobl^{2,p}(Q)$.

For $k\in\NN$, we 
let $\DD^k \df \DD(\RR_+,\RR^k)$ denote the space of 
$\RR^k$-valued c\'adl\'ag functions on $\RR_+$.
When $k=1$, we write $\DD$ for $\DD^k$. 

For a nonnegative function $g\in\Cc(\RR^{d})$ 
we let $\order(g)$ denote the space of functions
$f\in\Cc(\RR^{d})$ satisfying
$\sup_{x\in\RR^{d}}\,\frac{\abs{f(x)}}{1+g(x)}<\infty$.
We also let $\sorder(g)$ denote the subspace of $\order(g)$ consisting
of those functions $f$ satisfying
$\limsup_{\abs{x}\to\infty}\,\frac{\abs{f(x)}}{1+g(x)}=0$.

For a probability  measure $\mu$ in $\cP(\RR^d)$,
the space of Borel probability measures on $\RR^d$
under the Prokhorov topology, and a real-valued function
$f$ which is integrable with respect to $\mu$ we use the notation
$\mu(f)\df \int_{\Rd} f(x)\,\mu(\D{x})$.

\section{The model and assumptions}\label{S2}
We consider a controlled jump diffusion process 
$\process{X}$ taking values in the $d$-dimensional Euclidean space $\RR^d$ 
defined by
\begin{equation}\label{E-sde}
\D{X_t} \,\df\,
b(X_t,U_t)\,\D{t} + \upsigma(X_t)\,\D{W_t} + \D{L_t}
\end{equation}
with $X_0 = x\in\RR^d$.
All random processes in \cref{E-sde}
are defined on a complete probability space $(\Omega,\sF,\Prob)$. 
The process $\process{W}$ is a $d$-dimensional standard Wiener process,
and $\process{L}$ is a L\'evy process defined as follows.
Let $\widetilde\cN(\D{t},\D{z})$ denote a martingale measure
on $\RR^l_*=\RR^l\setminus \{0\}$, $l\ge1$,
taking the form $\widetilde\cN(\D{t},\D{z}) = \cN(\D{t},\D{z}) - \Pi(\D{z})\D{t}$, 
where $\cN$ is a Poisson random measure, 
and $\Pi(\D{z})\D{t}$ is the corresponding intensity measure,
with $\Pi$ a finite measure on $\RR^l_*$. 
Then, $\process{L}$ is given by
\begin{equation*}
\D L_t \df \int_{\RR^l_*} g(z)\, \widetilde\cN(\D t, \D z) 
\end{equation*}
for a measurable function $g\colon \RR^d\times \RR^l \to \RR^d$.   
The control process $\process{U}$ takes values in a compact,
metrizable space $\Act$, $U_t(\omega)$ is jointly measurable in
$(t,\omega)\in [0,\infty)\times\Omega$,
and is non-anticipative:
for $s < t$, $\bigl(W_{t} - W_{s},\, 
\cN(t,\cdot)-\cN(s,\cdot)\bigr)$ is independent of
\begin{equation*}
\sF_{s} \,\df\, \text{the completion of~}
\sigma\{X_{0}, U_{r}, W_{r},\cN(r,\cdot)\,\colon\, r\le s\}
\text{~relative to~} (\sF,\Prob)\,.
\end{equation*}
Such a process $U$ is called an admissible control, and we let $\Uadm$ denote the set
of admissible controls.
We also assume that the initial conditions $X_0$, $W_0$ and $\cN(0,\cdot)$ are independent.

To guarantee the existence of a solution to the equation \cref{E-sde}, 
we impose the following usual assumptions on the drift,
matrix $\upsigma$ and jump functions
(compare with \cite{ACPZ19}*{Section~4.2}). 
The functions $b \colon  \RR^d\times\Act \mapsto \RR^d$ and 
$\upsigma = [\upsigma^{ij}]\colon\RR^d\mapsto\RR^{d\times d}$ are continuous 
and have at most affine growth on $\RR^d$.
Also, $b$ is locally Lipschitz continuous in its first argument uniformly
with respect to the second.
The matrix $\upsigma$ is locally Lipschitz continuous  and nonsingular.
We also assume that $\int_{\RR^l_*} \abs{g(z)}^2\,\Pi(\D{z}) < \infty$.
Define $\nu(A) \df \Pi\bigl(\{z\in\RR_{*}^{l}\,\colon g(z)\in A\}\bigr)$.
Thus, $\nu$ is a Radon measure on $\RR^d$, and we
let $\bm{\nu} \df \nu(\RR^d) = \Pi(\RR^l_*)$, which is finite.
These hypotheses are enforced throughout the rest of the paper.

Under the above assumptions on the parameters, \cref{E-sde} has a unique strong solution
under any admissible control $U$
(see, e.g., \cite{GS72}*{Part~II, \S\,7}),
which is right continuous w.p.1, and has the strong Feller property.
Recall that \emph{Markov controls} may be
identified with Borel measurable map $v$ on $\RR_+\times\RR^d$, by letting
$U_t = v(t,X_t)$.
For any such Markov control $v$,
define the associated diffusion process $\{X^{\circ}, t\ge 0\}$ by
\begin{equation}\label{E-diffusion}
\D{X^{\circ}_t} \,\df\,
b(X^{\circ}_t,v(t,X^{\circ}_t))\,\D{t} + \upsigma(X^{\circ}_t)\,\D{W_t}\,,
\end{equation}
with $X^{\circ}_0 = x^\circ\in\RR^d$.
It is well known that \cref{E-diffusion}
has a pathwise unique strong solution \cite{Gyongy-96}*{Theorem~2.4}.
Since $\Pi$ is finite, it follows by the construction of a solution
in \cite{Skorokhod-89}*{Chap.~1, Theorem~14} via \cref{E-diffusion} 
that \cref{E-sde} has a unique strong solution under any Markov control.
We say that a Markov control $v$ is \emph{stationary} if $v(t,x)$ is
independent of $t$, and we use the symbol
$\Usm$ to denote the set of these controls.

For $\varphi\in \Cc^2(\RR^d)$, define the integro-differential operator 
$\Ag\colon \Cc^2(\RR^d) \to \Cc(\RR^d \times \Act)$ by
\begin{equation}\label{E-Ag}
\Ag \varphi(x,u) \,\df\,  a^{ij}(x)
\partial_{ij}\varphi(x) + \widetilde{b}^i(x,u)\partial_i \varphi(x) 
+ \int_\Rd  \bigl(\varphi(x+y)-\varphi(x)\bigr)\,\nu(\D y)\,,
\end{equation}
where $a\df\frac{1}{2}\upsigma\upsigma\transp$,
and $\widetilde{b}(x,u) \df b(x,u)+ \int_\Rd z\,\nu(\D z)$.
With $u\in\Act$ treated as a parameter, we also
define $\Ag_u \varphi(x)\df \Ag \varphi(x,u)$.
We decompose this operator as $\Ag_u=\widetilde\Lg_u + \widetilde\cI$, where
\begin{equation}\label{E-cI}
\widetilde\Lg_u\varphi(x) \,\df\, a^{ij}(x)
\partial_{ij}\varphi(x) + \widetilde{b}^i(x,u)\partial_i \varphi(x)
- \bm{\nu}\varphi(x)\,, \quad \text{and}\quad 
\widetilde{\cI}\varphi(x)\,\df\, \int_\Rd \varphi(x+y)\,\nu(\D y)\,.  
\end{equation}

Let $D$ be a bounded domain with $C^{1,1}$ boundary.
Recall that $\uptau(D)$ denotes the first exit time from $D$.
As shown in \cite{ACPZ19}*{Lemma~4.1}, for any $f\in\Sobl^{2,d}(\RR^d)$, 
such that $\widetilde{\cI}\abs{f}\in \Lpl^{d}(\RR^d)$, we have 
\begin{equation}\label{E-Ito}
\Exp_x^U[f(X_{t\wedge\uptau(D)})] \,=\, 
f(x) + \Exp_x^{U}\biggl[\int_0^{t\wedge\uptau(D)}\Ag f(X_s,U_s)\,\D{s}\biggr]
\end{equation}
for all $x\in D$, $t\ge0$, and $U\in\Uadm$.
In addition,  \cref{E-Ito} holds if we replace $t\wedge\uptau(D)$ with $\uptau(D)$.
Here, $\Exp^U_x$ denotes the expectation operator on the canonical space of the
process under the control $U\in\Uadm$.
Equation \cref{E-Ito} arises from the well known Krylov's extension of the
It\^{o}'s formula, and we refer to this plainly as the It\^{o} formula.

\subsection{The ergodic control problem}

Given a continuous \emph{running cost} function
 $\rc \colon \RR^d\times\Act \rightarrow \RR_+$, which
is locally Lipschitz continuous in its first argument
uniformly with respect to the second,
 we define the \emph{average} (or
\emph{ergodic}) penalty as
\begin{equation}\label{E-rhoU}
\varrho^{\vphantom{\frac{1}{2}}}_{U}(x)
\,\df\, \limsup_{T\rightarrow\infty}\,\frac{1}{T}\,
\Exp_x^U\biggl[\int_0^{T}\rc(X_t,U_t)\,\D{t}\biggr]\,.
\end{equation}
for an admissible control $U\in\Uadm$.
We say that $U\in\Uadm$ is \emph{stabilizing} if
$\varrho^{\vphantom{\frac{1}{2}}}_{U}(x)<\infty$ for all $x\in\Rd$.

The ergodic control problem seeks to minimize the ergodic penalty over
all admissible controls.
We define 
\begin{equation}\label{E-rho*}
\varrho_*(x)\,\df\,\inf_{U\in\Uadm}\,\varrho^{\vphantom{\frac{1}{2}}}_{U}(x)\,.
\end{equation}
As we show in \cref{T4.1}, the \emph{optimal ergodic value}
$\varrho_*$ does not depend on $x$.

\cref{A2.1} which follows,
is a slight variation of \cite{ABP15}*{Assumption 3.1},
and is abstracted from the limiting diffusions arising
in multiclass stochastic networks in the Halfin--Whitt regime.
Note that the assumption on the running cost in \cite{ACPZ19}*{Section 2.2}
is not met in these problems.
Recall that a function $f\,\colon \mathcal{X}\to\RR$, where $\mathcal{X}$ is
a $\sigma$-compact space, is called \emph{coercive}, or \emph{inf-compact}
if the set $\{ x\in\mathcal{X}\,\colon f(x)\le C\}$ is compact (or empty)
for every $C\in\RR$. 

\begin{assumption}\label{A2.1}
There exist some open set $\cK\subset\RR^{d}$, a ball $\sB_\circ$, 
and coercive nonnegative
functions $\Vo\in\Cc^{2}(\RR^{d})$ and $F \in \Cc(\RR^{d}\times\Act)$  
such that:
\begin{enumerate}
\item[(i)]
The running cost $\rc$ is coercive on $\cK$.
\item[(ii)]The following inequalities hold
\begin{equation}\label{EA2.1A}
\begin{split}
\Ag_u\Vo(x) &\,\le\, \Ind_{\sB_\circ}(x) - F(x,u)\qquad
\forall\,(x,u)\in \cK^{c}\times\Act\,,\\
\Ag_u\Vo(x) &\,\le\,\Ind_{\sB_\circ}(x) + \rc(x,u)
\qquad\forall\,(x,u)\in\cK\times\Act\,.
\end{split}
\end{equation}
\end{enumerate}
Without loss of generality, we assume $F$ is locally Lipschitz continuous
in its first argument.
\end{assumption}

Since we can always scale $\sB_\circ$,
$\Vo$ and $F$ to obtain the form in \cref{EA2.1A}, there is no
need to include any other constants in these equations. 
It is worth noting that $\rc$ is coercive on $\Rd$ 
if $\cK^c$ is bounded, and the controlled jump diffusion is uniformly stable
if $\cK$ is bounded. 

We introduce an additional assumption which, together with \cref{A2.1},
is sufficient for the existence
of a stabilizing stationary Markov control.
For $v\in\Usm$, we let $b_v(x)\df b\bigl(x,v(x)\bigr)$,
and define $\Ag_v$, $\widetilde\Lg_v$, $\rc_v$, and $\varrho_v$ analogously.
If under $v\in\Usm$ the controlled jump diffusion is positive recurrent,
then $v$ is called a \emph{stable} Markov control, 
and the set of such controls is denoted by $\Ussm$.

\begin{assumption}\label{A2.2}
There exist $\Hat{v}\in\Ussm$, a positive constant $\Hat{\kappa}$,
and a coercive nonnegative function 
$\sV\in\Cc^2(\RR^d)$ such that
\begin{equation}\label{EA2.2A}
\Ag_{\Hat{v}}\sV(x) \,\le\,
\Hat{\kappa}\Ind_{\sB_\circ}(x) - \rc_{\Hat{v}}(x) \,,\qquad\forall x\in\RR^d\,,
\end{equation}
with $\sB_\circ$ as in \cref{A2.1}.
\end{assumption}

Without loss of generality,
we use the same ball $\sB_\circ$ in \cref{A2.1,A2.2} in the interest
of notational economy.

\begin{remark}
The reader will note that \cref{A2.2} is not used in \cite{ABP15}.
Instead, starting from a weak stabilizability hypothesis, namely that
\begin{equation}\label{E-stab}
\varrho^{\vphantom{\frac{1}{2}}}_{U}(x)\,<\,\infty\quad\text{for some\ }x\in\Rd
\text{\ and\ } U\in\Uadm\,,
\end{equation}
the existence of a control $\Hat{v}\in\Ussm$ and a coercive nonnegative function
$\sV\in\Cc^2(\RR^d)$ satisfying \cref{EA2.2A} is established
in \cite{ABP15}*{Lemma~3.1}.
For the model studied in this paper, if we assume \cref{E-stab},
then together with \cref{A2.1} we can show, that there exists
a control $\Hat{v}$ which is stabilizing for some coercive running cost
$\widetilde\rc\ge\rc$ (see the proof of \cref{T4.1} which appears later).
Then, if $\nu$ has compact support, 
\cite{ACPZ19}*{Theorem~3.7} shows that there exists a function
$\sV\in\Sobl^{2,p}(\RR^d)$, for any $p>1$, satisfying \cref{A2.2},
and this implies that $\widetilde\cI\sV\in \Lpl^{d}(\RR^d)$.
Thus, if $\nu$ has compact support,
then the It\^o formula in \cref{E-Ito} is applicable to $\sV$, and using
this in the proofs, it follows
that as far as the results of this paper are concerned, we may replace
\cref{A2.2} with the weaker hypothesis in \cref{E-stab},
which cannot be weakened further since it is necessary for the
value of the ergodic control problem to be finite.
In typical applications, the existence of a stabilizing Markov
control is usually established by exhibiting
 a Foster--Lyapunov equation taking the form of \cref{EA2.2A}.

As we establish in \cref{T4.1}, \cref{A2.1} and \cref{E-stab} together
guarantee the existence of an optimal stationary Markov control
for the ergodic control problem.
Thus \cref{A2.2} need not be used for the existence part. 
However, it plays a crucial role in the derivation of
the HJB equation in \cref{S5} for non-compactly supported $\nu$.
\end{remark}

\section{Examples}\label{S3}

In this section, we provide examples of stochastic networks, and 
show that the jump diffusions involved satisfy \cref{A2.1,A2.2}.
We refer the reader to 
\cite[Section 2]{AP16a} for a detailed description of multiclass multi-pool networks.

Consider a multiclass multi-pool network with $d$ classes of customers and $J$ server pools.
Define the sets $\sI \df \{1,\dots,d\}$, $\sJ \df \{1,\dots,J\}$, and
\begin{equation*}
\Act \,\df\, \bigl\{u = (u^c,u^s)\in\RR^d_+\times\RR^J_+ \colon 
\langle e,u^c \rangle = \langle e,u^s \rangle = 1 \bigr\}\,.
\end{equation*}
Following similar arguments as in \cite[Theorem 2.1]{APZ19b}, and assuming that
service interruptions are asymptotically negligible under the $\sqrt{n}$-scaling,
we can show that the limiting controlled queueing processes are 
$d$-dimensional jump diffusions taking the form 
\begin{equation}\label{ES3A}
\D X_t \,=\, b(X_t,U_t)\,\D{t} + \sigma\,\D{W}_t +  \theta\,\D{L}_t\,,
\end{equation} 
where
$\sigma$ is a nonsingular diagonal matrix, $\theta$ is a strictly positive vector, and 
$\process{L}$ is a one-dimensional compound Poisson process.
The L\'evy measure of $\theta L_t$ is denoted by $\nu(\D{z})$.
This is supported on $\{\theta t\colon t\in[0,\infty)\}$.
It follows by \cite[Lemma 4.3]{AP16a} that 
\begin{equation}\label{ES3B}
b(x,u) \,=\, \ell - M_1\bigl(x - \langle e,x \rangle^+u^c\bigr)
- \langle e,x \rangle^+\varGamma u^c + \langle e,x \rangle^-M_2u^s 
\end{equation} 
where $\ell\in\RR^d$, $\varGamma = \diag(\gamma_1,\dots,\gamma_d)$, $M_1$ is a lower-diagonal $d\times d$
matrix with positive diagonal elements, and $M_2$ is a $d\times J$ matrix.
Without loss of generality, we assume that $\gamma_1 = 0$, $\gamma_d > 0$, and $\gamma_i \ge 0$, $i\in\sI\setminus\{1,d\}$.
We consider the ergodic control problem  in \cref{E-rho*} with
\begin{equation}\label{ES3C}
\rc(x,u) \,\df\, \sum_{i\in\sI}c_i[\langle e,x \rangle^+u^c_i]^m + 
\sum_{j\in\sJ}s_j[\langle e,x \rangle^-u^s_j]^m
\end{equation} for some $m\ge1$, 
and some positive constants $\{c_i\colon i\in\sI\}$ and $\{s_j\colon j\in\sJ\}$. 
This running cost function penalizes the queue sizes and idleness.
It is evident that $\rc(x,u)$ is not near-monotone,
since $\langle e,x \rangle$ equals $0$ on a hyperplane in $\RR^d$.
We assume that $\int_{\RR^d}\abs{z}^{m}\,\nu(\D{z})<\infty$.

We define
$\cK_\delta \,\df\, \{x\in\RR^d\colon \abs{\langle e,x \rangle} > \delta\abs{x}\}$
with $\delta > 0$.
It is clear that $\rc$ is coercive on $\cK_\delta$ for $\delta > 0$.
For a positive definite symmetric matrix $Q$,
we let $g(x)$ be some positive convex smooth function which 
agrees with $\langle x,Qx\rangle^{\nicefrac{1}{2}}$ on $B_1^c$,
and define the function $\sV_{Q,k}(x) = \bigl(g(x)\bigr)^k$
for $k >0$.

\begin{lemma}\label{L3.1}
There exist a diagonal matrix $Q$,
some $\delta>0$ small enough, 
and a positive constant $C$ such that
$\sV_\circ = \sV_{Q,m}$ and $F(x) = C\abs{x}^m$
satisfy \cref{A2.1} with $\cK=\cK_\delta$.
\end{lemma}

\begin{proof}
Recall $\widetilde{b}$ defined in \cref{E-Ag}. 
Following the same calculation as in the proof of \cite[Theorem 4.1]{AP16a},
we obtain
\begin{equation*}
\bigl\langle \widetilde{b}(x,u),\grad \sV_{Q,m}(x)\bigr\rangle \,\le\,
\begin{cases}
C_1 - m \langle x,Qx \rangle^{\nicefrac{m}{2}-1}\abs{x}^2 &
\forall\,(x,u)\in\cK_\delta^c\times\Act\,,\\[5pt]
C_1 \bigl(1 + \abs{\langle e,x \rangle}^m\bigr)
& \forall\,(x,u)\in\cK_\delta\times\Act
\end{cases} 
\end{equation*}
for some $\delta>0$, a positive constant $C_1$,
and a diagonal matrix $Q$ satisfying
$x\transp(QM_1 + M_1\transp Q)x \ge 8 \abs{x}$.
On the other hand, using the hypothesis
$\int_{\RR^d}^{}\abs{z}^m\,\nu(\D{z}) <\infty$,
we obtain
\begin{equation}\label{PL3.1A}
\begin{aligned}
\int_{\RR^d}\bigl(\sV_{Q,m}(x + z) - \sV_{Q,m}(x)\bigl)\,\nu(\D{z})
&\,=\,
\int_{\RR^d} \int_0^{1}
\bigl\langle z, \grad \sV_{Q,m}(x + tz)\bigr\rangle\,\D{t}\,\nu(\D{z}) \\
&\,\le\, C_2 + \epsilon \langle x,Qx \rangle^{\nicefrac{m}{2}}
\end{aligned}
\end{equation}
for some $\epsilon>0$ sufficiently small, 
and a positive constant $C_2$.
Thus \cref{EA2.1A} holds. This completes the proof.
\end{proof}

\begin{remark}
Let $\Tilde{\ell} \df \ell + \int_{R^d}z\,\nu(\D{z})$ and $u^c_1 = 1$,
and suppose that $\langle e,(M_1^{-1})\transp\Tilde{\ell} \rangle > 0$.
Using the leaf elimination algorithm as in \cite[Theorem 4.2]{AP16a},
we obtain a constant control $\bar{u} = (\Bar{u}^c,\Bar{u}^s) \in\Act$,
with $\Bar{u}^c_1 = 1$, such that the last two terms on the right hand side of
\cref{ES3B} are equal to $0$.
This implies that $\process{X}$ is transient under the control $\Bar{u}$
by \cite[Theorem 3.1]{APS19}.
Therefore, \cref{ES3A} is not uniformly stable.

Recall $\widetilde\Lg$ and $\widetilde{\cI}$ defined in \cref{E-cI}.
By \cite[Theorem 4.2]{AP16a} concerning the local operator $\widetilde\Lg$,
and \cref{PL3.1A} for $\widetilde{\cI}$,
it follows that there exist $u = (u^c,u^s)\in\Act$ with $u^c_d = 1$, and 
$\sV(x)\sim\langle x,\tilde{Q}x \rangle^{\nicefrac{m}{2}}$ for some 
diagonal positive matrix $\tilde{Q}$
satisfying \cref{A2.2}.
\end{remark}

We present two specific examples: the `W' and `V' networks. 

\begin{example}
(The `W' model with service interruptions.)
See \cite[Section 4.2]{AP16a} for the detailed definition of the `W' model.
We have $\sI = \{1,2,3\}$ and $\sJ = \{1,2\}$.
By \cite[Example 4.2]{AP16a}, $M_1$ and $M_2$ in \cref{ES3B} are given by
\begin{equation*}
M_1 \,=\, \begin{bmatrix}
\mu_{11} & 0 & 0 \\
\mu_{22}-\mu_{21} & \mu_{22} & 0\\
0 & 0 & \mu_{32} 
\end{bmatrix}\,, \quad
M_2 \,=\, \begin{bmatrix}
0 & 0 \\
\mu_{21} - \mu_{22} & 0\\
0 & 0
\end{bmatrix}
\end{equation*}
for some positive constants $\{\mu_{ij}\colon i\in\sI,j\in\sJ,(i,j)\notin\{(1,2),(3,1)\}\}$.
We assume that $\gamma_1 = \gamma_2 = 0$ and $\gamma_3=1$, 
and $\langle e,(M_1^{-1})\transp\Tilde{\ell} \rangle > 0$.
By \cite[Theorem 3.1]{AHP19}, under any control $v\in\Usm$ with $v_3=0$ and $v_5 = 1$, 
$\{X_t\}_{t\ge0}$ is transient.
On the other hand,  \cref{A2.2} is satisfied
for the constant control $u^c_3 = 1$ and $u^s_2 = 1$.
\end{example}

\begin{example}
(The `V' model with service interruptions.)
\Cref{ES3A} also describes
the limiting jump diffusions of the `V' model.
Here $\sI = \{1,\dots,d\}$, $\sJ = \{1\}$, and
\begin{equation*}
b(x,u) \,=\, \ell - M\bigl(x - \langle e,x \rangle^+u\bigr)
- \langle e,x \rangle^+\varGamma u\,,
\end{equation*}
where $u$ takes values $\Act = \{u\in\RR^d\colon \langle e,u \rangle = 1 \}$, and
$M = \diag(\mu_1,\cdots,\mu_d)$ is a positive diagonal matrix. 
Suppose that there exists a nonempty set $\sI_0\subset\{1,\cdots,d-1\}$
such that $\gamma_i = 0$ for $i\in\sI_0$, and
$\langle e,M^{-1}\tilde{\ell} \rangle > 0$.
In this case, \cite[Theorem 3.3]{APS19} asserts that $\process{X}$ is transient
under any $v\in\Ussm$ satisfying $\varGamma v = 0$.
However, \cref{A2.1} is satisfied by \cite[Remark 5.1]{APS19}, and,
provided that $\gamma_i>0$ for some $i\in\sI$, then
\cref{A2.2} holds by \cite[Theorem 3.5]{APS19}.
\end{example}

\begin{remark}
It is shown in \cite{AHP18} that the limiting diffusion of the `V' model 
without service interruptions
is uniformly ergodic over all stationary Markov controls,
if either $\varGamma >0$, or the \emph{spare capacity} $-\langle e,M^{-1}\ell \rangle$
is positive. 
This result has been extended to the limiting jump diffusion 
of the `V' model with service interruptions in \cite{AHPS19}, with the
difference that uniform ergodicity is over all stationary Markov 
controls resulting in a locally Lipschitz continuous drift.
It is also shown in \cite{AHP19} that if the spare capacity is positive, 
then the limiting diffusion of the multiclass multi-pool networks with a dominant server pool
(for example the `N' and `M' models), or class-dependent 
service rates, is uniformly exponentially
ergodic over all stationary Markov controls.
However, in general, multiclass multi-pool  
networks do not enjoy uniform ergodicity, but fall in the framework 
of \cref{A2.1,A2.2}.  	
\end{remark}

\section{Existence of an optimal stationary Markov control}\label{S4}

In this section
we establish the existence of an optimal stationary Markov control
by following a standard convex analytic argument.
We adopt the relaxed control framework (see, e.g., \cite{ABG12}*{Section~2.3}),
and extend the definitions of $b$ and $\rc$ accordingly, that is we let
$b_v(x) = \int_{\Act}b(x,u)\,v(\D{u}\,|\,x)$, where $v(x) = v(\D{u}\,|\,x)$ 
is a measurable kernel on $\Act$ given $x$, and analogously for $\rc$.
Let $\mu_v\in\cP(\RR^d)$ denote the unique invariant probability measure of
\cref{E-sde} under $v\in\Ussm$. 
Define the corresponding
\emph{ergodic occupation measure} $\uppi_v \in \cP(\RR^d\times\Act)$ by 
$\uppi_v(\D{x},\D{u}) \df \mu_v(\D{x})\,v(\D{u}\,|\,x)$. 
The class of all ergodic occupation measures is denoted by $\eom$.
Let $\Cc_0^2(\RR^d)$ denote the Banach space of functions $f\colon\Rd\to\RR$
that are twice continuously differentiable and their derivatives up to
second order vanish at infinity, and $\sC$ denote some fixed dense subset of
$\Cc_0^2(\RR^d)$ consisting of functions with compact supports.
Applying the Theorem~in \cite{PE82}, it follows that $\uppi\in\eom$ if and only if
\begin{equation*}
\int_{\RR^d}^{}\Ag_u f(x)\,\uppi(\D{x},\D{u}) \,=\, 0 \qquad\forall\,f\in\sC\,.
\end{equation*}
It is easy to show that $\eom$ is a closed and convex subset of $\cP(\RR^d\times\Act)$ 
(see, e.g., \cite{ABG12}*{Lemma~3.2.3}).

Recall also the definition of empirical measures.

\begin{definition}\label{D4.1}
For $U \in \Uadm$ and $x \in \RR^d$, we define the mean empirical measures
$\{\Bar{\zeta}^U_{x,t}\,\colon t>0\}$, and (random) empirical measures
$\{\zeta^U_{t}\,\colon t>0\}$ by
\begin{equation*}
\Bar{\zeta}^U_{x,t}(f)\,=\,
\int_{\RR^d\times\Act}f(x,u)\,\Bar{\zeta}^U_{x,t}(\D{x},\D{u})
\,\df\, \frac{1}{t}\int_{0}^{t}
\Exp_x^U\biggl[\int_{\Act} f(X_s,u)\,U_s(\D{u})\biggr]\,\D{s}\,,
\end{equation*}
and
\begin{equation*}
\zeta^U_{t}(f)\,=\,
\int_{\RR^d\times\Act}f(x,u)\,\D \zeta^U_{t}(\D{x},\D{u})\,\df\, \frac{1}{t} 
\int_0^t\int_{\Act} f(X_s,u)\,U_s(\D{u})\,\D s\,,
\end{equation*}
respectively, for all $f\in\Cc_b(\RR^d\times\Act)$.
\end{definition}

Let $\overline{\RR}^{d}$ denote the one-point compactification of $\RR^{d}$.
Then as shown in \cite{ACPZ19}*{Lemma~4.2},
every limit $\Hat{\zeta} \in  \cP(\overline{\RR}^d\times\Act)$
of $\zeta^Z_{x,t}$ as $t\to\infty$ takes the form 
$\Hat{\zeta}= \delta \zeta' + (1-\delta) \zeta''$ for some $\delta \in [0,1]$,
with $\zeta' \in \eom$ and $\zeta'' (\{\infty\} \times \Act) =1$ almost surely.
The same claim  holds for the mean empirical measures,
without the qualifier `almost surely'.

We borrow the technique introduced in \cite{ABP15}.
Recall the function $F$ and the set $\cK$ in \cref{A2.1}.
First, define the set
\begin{equation*}
\widetilde{\cK} \,\df\,
(\cK\times\Act)\cup\bigl\{(x,u)\in\RR^d\times\Act\,\colon\rc(x,u) > F(x,u)\bigr\}\,.
\end{equation*}
We have
\begin{equation}\label{E4.1A}
\Ag_u\Vo(x,u) \,\le\, \Ind_{\sB_{\circ}}(x,u) - F(x,u)\Ind_{{\widetilde{\cK}^c}}(x,u)
+ \rc(x,u)\Ind_{{\widetilde{\cK}}}(x,u)\,, \quad \forall (x,u)\in\RR^d\times\Act\,.
\end{equation}
As shown in \cite{ABP15}*{Lemma~3.3}, there exists a coercive
function $\widetilde{F}\in\Cc(\RR^d\times\Act)$,
which is locally Lipschitz in its first argument,
and satisfies
\begin{equation}\label{E-tildF}
\rc \,\le\, \widetilde{F} \,\le\, \Tilde{\kappa}\, \bigl(\Ind_{\sB_\circ}
+ \rc\,\Ind_{\widetilde{\cK}} + F\,\Ind_{{\widetilde{\cK}^c}}\bigr)
\end{equation}
for some positive constant $\Tilde{\kappa}\ge1$.
Here again we select the same ball $\sB_\circ$ as in \cref{A2.1} for convenience.
This can always be accomplished by adjusting the constant $\Tilde\kappa$.

Define the perturbed running cost
$\rc^{\epsilon} \df \rc + \epsilon\widetilde{F}$.
Since $\rc^\epsilon$ is coercive for $\epsilon>0$,
the results of \cite{ACPZ19} are applicable for the ergodic control
problem with the perturbed running cost.
At the same time, it follows from \cref{E4.1A,E-tildF} and
the argument in the proof of \cite[Theorem~3.1]{ABP15},   
that if a control $U\in\Uadm$ is stabilizing for $\rc$,  then it is also
stabilizing for $\rc^\epsilon$ for any $\epsilon>0$.

\begin{theorem}\label{T4.1}
Grant \cref{A2.1}.
Then every stabilizing stationary Markov control is in $\Ussm$. 
In addition, if the stabilizability hypothesis in \cref{E-stab}
is met, then
there exists a stationary Markov control which is optimal
for the ergodic control problem, and $\varrho_*$ is a constant.
\end{theorem}

\begin{proof}
By \cref{A2.1}, we have $\widetilde{\cI}\sV_\circ \in \Lpl^{\infty}(\RR^d)$, 
and thus, applying It\^o's formula and Fatou's lemma to
\cref{E4.1A}, it follows by \cref{E-tildF} that
\begin{equation}\label{PT4.1A}
\Bar{\zeta}^U_{x,t}(\rc^\epsilon)
\,\le\, \Bar{\zeta}^U_{x,t}(\rc) + \epsilon\Tilde\kappa\biggl(1 + \frac{1}{t}\Vo(x)
+ 2\Bar{\zeta}^U_{x,t}(\rc)\biggr)
\qquad\forall\,(x,t)\in\Rd\times(0,\infty)\,,\ \forall\,U\in\Uadm\,.
\end{equation}
Since by \cref{PT4.1A} we have
\begin{equation}\label{PT4.1B}
\uppi_v(\rc^\epsilon)\,\le\,\varrho_v
+\epsilon \Tilde\kappa(1+2\varrho_v)
\end{equation}
for any stabilizing stationary Markov control $v$, we have $\uppi_v(\rc) <\infty$,
and the first assertion follows.

Define $\varrho^\epsilon_U$ and
$\varrho^\epsilon_*$ as in \cref{E-rhoU,E-rho*}, respectively, by
replacing $\rc$ with ${\rc}^\epsilon$.
Let $\Hat\varrho^\epsilon_*\df\inf_{\uppi\in\eom}\,\uppi(\rc^\epsilon)$,
and $\Hat\varrho_*\df\inf_{\uppi\in\eom}\,\uppi(\rc)$.
Since
$\rc^\epsilon$ is coercive for any $\epsilon\in(0,1)$, we have
 $\Hat\varrho^\epsilon_*=\uppi_{v_*^\epsilon}(\rc^\epsilon)$ for some
$v_*^\epsilon\in\Ussm$ by \cite{ABG12}*{Theorem~3.4.5}, and
$\Hat\varrho^\epsilon_*=\varrho^\epsilon_*$ by 
Lemma~4.2 in \cite{ACPZ19}  and the proof of \cite{ABG12}*{Theorem~3.4.7}.
Hence, by \cref{PT4.1A}, which implies
that  $\varrho^\epsilon_U\,\le\,\varrho_U +\epsilon \Tilde\kappa(1+2\varrho_U)$,
and the above definitions we have
\begin{equation*}
\varrho_*\,\le\,\Hat\varrho_*\,\le\,\Hat\varrho^\epsilon_* \,=\, \varrho^\epsilon_*
\,\le\, \varrho_* +\epsilon \Tilde\kappa(1+2\varrho_*)
\qquad\forall\,\epsilon\in(0,1) \,.
\end{equation*}
This shows that $\varrho_*=\Hat\varrho_*$.
It remains to show that $\varrho_*=\uppi_{v_*}(\rc)$ for some $v_*\in\Ussm$.
But this follows by using the technique in the proof of  \cite{ABG12}*{Theorem~3.4.5}.
This completes the proof.
\end{proof}

\section{The HJB equations}\label{S5}

In this section, we study the $\alpha$-discounted and ergodic HJB equations 
for the jump diffusion defined in \cref{E-sde}.
For the $\alpha$-discounted control problem,  
it is rather standard to establish the existence of solutions and 
the characterization of optimal controls
(see \cref{T5.1} below for details).
We consider the Dirichlet problem on $B_R$
for the $\alpha$-discounted
problem with running cost $\rc^{\epsilon}$.
From \cite{BenLi-84}*{Chap.~3, Theorem~2.3 and Remark~2.3},
there exists a unique solution 
$\psi^{\epsilon}_{\alpha,R}\in\Sob^{2,p}(B_R)\cap\Sob^{1,p}_0(B_R)$
to the (homogeneous) Dirichlet problem
\begin{equation}\label{E-dir}
\min_{u\in\Act}\,\bigl[\Ag_u\psi^{\epsilon}_{\alpha,R}
+ \rc^{\epsilon}(\cdot,u)\bigr] \,=\,
\alpha\psi^{\epsilon}_{\alpha,R} \quad \text{in } B_R\,, \quad \text{and} \quad
\psi^{\epsilon}_{\alpha,R} \,=\, 0 \quad \text{in } B^{c}_R\,.
\end{equation} 
For the Dirichlet problem with a linear integro-differential operator,
existence and uniqueness of a solution are also asserted in
\cite{GM02}*{Theorem~3.1.22}. 
Meanwhile, for a bounded running cost
function \cite{Menaldi-99}*{(1.26)} and
under the blanket stability assumption in \cite{Menaldi-99}*{(1.6)},
HJB equations on the whole space are established in
\cite{Menaldi-99}*{Remark~3.3 and Theorem~4.1}. 
It is clear that these assumptions are not met for multiclass stochastic networks in
the Halfin--Whitt regime. 
For example, in \cref{ES3C}, the running cost function penalizing the queueing and 
idleness is unbounded, and the drift in \cref{ES3A}
does not satisfy \cite{Menaldi-99}*{(1.6)}.

\begin{theorem}\label{T5.1}
Grant \cref{A2.1,A2.2}.
Then for any $\alpha \in (0,1)$ and $\epsilon\in[0,\Tilde{\kappa}^{-1})$,
the function $\psi^\epsilon_{\alpha,R}$ in \cref{E-dir} converges uniformly on
compacta to a function $V^{\epsilon}_\alpha\in\Sobl^{2,p}(\RR^d)$ for any $p>1$,
which is the minimal nonnegative solution of the HJB equation
\begin{equation}\label{ET5.1A}
\min_{u\in\Act}\,\bigl[\Ag_u V^{\epsilon}_{\alpha}(x) + \rc^{\epsilon}(x,u)\bigr] \,=\, 
\alpha V^{\epsilon}_{\alpha}(x)\quad \text{a.e.\ in\ } \Rd\,,
\end{equation}
and has the stochastic representation
\begin{equation}\label{ET5.1B}
V^{\epsilon}_{\alpha}(x) \,=\, \inf_{U\in\Uadm}\,
\Exp^U_x\biggl[\int_0^{\infty}\E^{-\alpha t}\rc^{\epsilon}(X_t,U_t)\,\D{t}\biggr]\,.
\end{equation}
In addition, a control $v\in\Usm$ is optimal, that is, it attains the infimum
in \cref{ET5.1B}, if and only if 
it is an a.e.\ measurable selector from the minimizer of \cref{ET5.1A}.
\end{theorem}

\begin{proof}
Under \cref{A2.2}, the proof for
the existence of a minimal nonnegative solution
$V^{\epsilon}_\alpha\in\Sobl^{2,p}(\RR^d)$ 
is exactly same as in \cite{ACPZ19}*{Theorem~3.2}.
A straightforward application of the comparison principle shows
that the following bound holds
\begin{equation}\label{PT5.1A}
V_\alpha^{\epsilon}(x) \,\le\, \frac{3\Hat{\kappa}+2}{\alpha} + \Vo(x) +
3\sV(x) \qquad \forall x\in\RR^d\,,\ \forall \alpha\in(0,1)\,,\
\forall \epsilon\in[0,\Tilde{\kappa}^{-1})\,. 
\end{equation}
From \cref{PT5.1A}, 
we have $\widetilde{\cI}V^{\epsilon}_\alpha \in \Lpl^{\infty}(\RR^d)$. 
Thus using the It\^{o}'s formula in \cref{E-Ito}, 
the stochastic representation and the sufficiency part of
the verification of optimality are established in
a standard manner (see, e.g., \cite{ABG12}*{Theorem~3.5.6 and Remark~3.5.8}).
On the other hand, for any $v\in\Usm$, 
the resolvent of the controlled diffusion defined in \cref{E-diffusion}
has a positive density with respect to the Lebesgue measure
by \cite{ABG12}*{Theorem~A.3.5}.
Since the L\'{e}vy measure $\nu$ is finite, 
then applying \cite{Li03}*{Lemma~2.1}, we see that
the same holds for the resolvent of the jump diffusion in \cref{E-sde}.
Thus, we may repeat the argument in  \cite{ABG12}*{Theorem~3.5.6} to establish the 
necessity part of the verification of optimality.
This completes the proof.
\end{proof}

We proceed to derive the HJB equation on the ergodic control problem by using the
vanishing discount method.
The technique used has some important differences from \cite{ACPZ19},
since here the running cost is not near-monotone when $\epsilon=0$.
To overcome this difficulty, we derive lower and upper bounds for $V^{\epsilon}_\alpha$
in the lemma which follows.

\begin{lemma}\label{L5.1}
Grant the hypotheses in \cref{A2.1,A2.2}.
For any $\delta\in (0, \frac{1}{2}]$, there exists $\tilde{r} = r(\delta) > 0$ such that 
\begin{equation}\label{EL5.1A}
V^\epsilon_\alpha \,\ge\, \inf_{B_r}\,V^\epsilon_{\alpha}
- \delta\Vo\qquad\text{on\ }B_r^c\,, \quad\forall\, r>\Tilde{r}\,,
\end{equation}
for all $\alpha\in(0,1)$ and $\epsilon\in[0,\Tilde\kappa^{-1})$.
Moreover, there exists $r_\circ > 0$ such that
\begin{equation}\label{EL5.1B}
V^{\epsilon}_{\alpha} \,\le\, 
\sup_{B_{r_\circ}}\, V^{\epsilon}_{\alpha} + \Vo
+ 3\sV \qquad \text{on\ } \RR^d\,,
\end{equation}
for all $\alpha\in(0,1)$ and $\epsilon\in[0,\Tilde\kappa^{-1})$. 
\end{lemma}

\begin{proof}
Let $v_*$ be an optimal control in $\Ussm$. Its existence has
been asserted in \cref{T4.1}.
Recall that $\mu_{v_*}$ denotes the invariant probability measure under $v_*$. 
Using \cref{ET5.1B}, Fubini's theorem, and \cref{PT4.1B}, we obtain
\begin{equation*}
\mu_{v_*}(B_r)\Bigl(\inf_{B_r}\,\alpha V^\epsilon_\alpha\Bigr) 
\,\le\, \varrho_* \,\le\,
\varrho_* +\epsilon \Tilde\kappa(1+2\varrho_*)
\end{equation*}
for any $r > 0$.
Fix some $r_\circ>0$ such that $B_{r_\circ}\supset\sB_\circ$.
Then
\begin{equation}\label{PL5.1A}
\inf_{B_r}\,\alpha V^\epsilon_\alpha \,\le\,
\frac{\epsilon\Tilde\kappa+ (1+2\epsilon\Tilde\kappa)\varrho_*}
{\mu_{v_*}(B_r)} \,\le\, \frac{\varrho_*}{\mu_{v_*}(B_{r_\circ})}\,,
\end{equation}
for all $r>r_\circ$,
$\alpha\in(0,1)$, and $\epsilon\in[0,\Tilde{\kappa}^{-1})$.

We first establish a lower bound of $V^\epsilon_\alpha$.
Let $\psi^\epsilon_{\alpha,R}$ satisfy \cref{E-dir},
and $\Hat{v}_R\in\Uadm$ be a measurable selector from its minimizer, that is,
it satisfies
\begin{equation}\label{PL5.1B}
\Ag_{\Hat{v}_R} \psi^\epsilon_{\alpha,R}-\alpha \psi^\epsilon_{\alpha,R}
\,=\, - \rc^\epsilon_{\Hat{v}_R}  \qquad \text{on\ } B_R\,.
\end{equation}
Let $\delta\in\bigl(0,\frac{1}{2}\bigr]$ be arbitrary.
By \cref{PL5.1A}, and the coerciveness of $\widetilde{F}$ in \cref{E-tildF},
there exists $\Tilde{r}=\Tilde{r}(\delta)>r_\circ$ such that
\begin{equation}\label{PL5.1C}
\inf_{B_{\Tilde{r}}}\, \alpha \psi^\epsilon_{\alpha,R}
\,\le\,\delta \Tilde\kappa^{-1}\widetilde{F}_{\Hat{v}_R}(x)
\quad\text{for all\ } x\in B_{\Tilde{r}}^c\,,\ R\ge \Tilde{r}\,,\ \alpha\in(0,1)\,,
\text{\ and\ }\epsilon\in[0,\Tilde{\kappa}^{-1})\,.
\end{equation}
Let
\begin{equation}\label{PL5.1D}
\phi^{\epsilon}_{\alpha,R} \,\df\, \delta\Vo + \psi^\epsilon_{\alpha,R}
-\inf_{B_{\Tilde{r}}}\,\psi^\epsilon_{\alpha,R}\,.
\end{equation}
By \cref{E-tildF,E4.1A,PL5.1B,PL5.1C}, we have 
\begin{equation}\label{PL5.1E}
\begin{aligned}
\Ag \phi^{\epsilon}_{\alpha,R} -\alpha \phi^{\epsilon}_{\alpha,R} &\,\le\,
\inf_{B_{\Tilde{r}}}\, \alpha \psi^\epsilon_{\alpha,R}
- \delta F_{\Hat{v}_R}\Ind_{\widetilde{\cK}^c}
- (1 - \delta) \rc_{\Hat{v}_R}\Ind_{\widetilde{\cK}}  \\
&\,\le\, \inf_{B_{\Tilde{r}}}\, \alpha \psi^\epsilon_{\alpha,R}
- \delta \Tilde\kappa^{-1}\widetilde{F}_{\Hat{v}_R}\\
&\,\le\, 0
\quad \text{on\ }B_R\setminus B_{\Tilde{r}}\,,\text{\ for all\ }R\ge \Tilde{r}\,.
\end{aligned}
\end{equation}
Since $\psi^\epsilon_{\alpha,R}$ converges monotonically
to $V^{\epsilon}_{\alpha}$ as $R\rightarrow\infty$ and $\Vo$ is coercive,
there exists $R_0=R_0(\delta,\alpha)>\Tilde{r}$ such that
\begin{equation}\label{PL5.1F}
\inf_{B_{\Tilde{r}}}\,\psi^\epsilon_{\alpha,R}
\,\le\,\delta\Vo(x)
\qquad\forall\,x\in B_R\setminus B_{R_{0}}\,,\ R>R_0\,.
\end{equation}
Thus,
since $\phi^{\epsilon}_{\alpha,R}\ge0$ on $B_{\Tilde{r}}$ by \cref{PL5.1D},
and $\phi^{\epsilon}_{\alpha,R}\ge0$ on $B_R\setminus B_{R_{0}}$ by \cref{PL5.1F},
it follows that $\phi^{\epsilon}_{\alpha,R}\ge0$ on $\Rd$ for all $R>R_0$ by
\cref{PL5.1E} and the strong maximum principle. 
Taking limits as $R\rightarrow\infty$ in
\cref{PL5.1D}, we obtain
\begin{equation*}
V^\epsilon_\alpha \,\ge\, \inf_{B_{\Tilde{r}}}\,V^\epsilon_{\alpha}
- \delta\Vo\qquad\text{on\ }B^c_{\tilde{r}}\,,
\end{equation*}
which establishes \cref{EL5.1A}.

Next we prove the upper bound. 
For $\Hat{v}$ in \cref{A2.2}, we have
\begin{equation}\label{PL5.1G}
\Ag_{\Hat{v}} (-\psi^\epsilon_{\alpha,R})-\alpha (-\psi^\epsilon_{\alpha,R})
\,\le\, \rc_{\Hat{v}} + \epsilon \widetilde{F}_{\Hat{v}} \qquad \text{on\ } B_R\,.
\end{equation}
Recall that $B_{r_\circ}\supset\sB_\circ$, and
select any balls $D_1$ and $D_2$, such that $B_{r_\circ}\Subset D_1\Subset D_2$.
By \cref{PL5.1B,E-tildF,E4.1A,PL5.1G},
the function
\begin{equation*}
\Hat{\phi}^{\epsilon}_{\alpha,R} \,\df\, 
\sup_{B_{r_\circ}}\,\psi^{\epsilon}_{\alpha,R}
- \psi^{\epsilon}_{\alpha,R} + \Vo + 3\sV
\end{equation*} 
satisfies
\begin{equation*}
\Ag_{\Hat{v}}\Hat{\phi}^{\epsilon}_{\alpha,R} - \alpha \Hat{\phi}^{\epsilon}_{\alpha,R}
\,\le\, -\sup_{B_{r_\circ}}\,\psi^{\epsilon}_{\alpha,R} 
\,\le\, 0 \qquad \text{on\ }B_R\setminus B_{r_\circ}\,,
\end{equation*}
for all $\alpha\in(0,1)$ and $\epsilon\in[0,\Tilde\kappa^{-1})$.
It is evident that $\Hat{\varphi}^{\epsilon}_{\alpha,R} \ge 0$
on $B_{r_\circ}\cup B_R^c$.
Thus, employing the strong maximum principle,
we obtain
\begin{equation}\label{PL5.1H}
\psi^{\epsilon}_{\alpha,R} \,\le\, 
\sup_{B_{r_\circ}} \psi^{\epsilon}_{\alpha,R} + \Vo
+ 3\sV \qquad \text{on\ } \RR^d\,,
\end{equation}
for all $\alpha\in(0,1)$ and $\epsilon\in[0,\Tilde\kappa^{-1})$.
Letting $R\rightarrow\infty$ in \cref{PL5.1H}, we obtain \cref{EL5.1B}.
This completes the proof.
\end{proof}

We also need the following estimate.
Its proof combines the technique in the proof of \cite{ACPZ19}*{Theorem~3.3}
with \cref{L5.1}.

\begin{lemma}\label{L5.2}
Grant the hypotheses in \cref{A2.1,A2.2}.
For each $R>0$, there exists a constant $\kappa_R$ such that
\begin{equation*}
\osc_{B_R}\,V^{\epsilon}_{\alpha} \,\le\, \kappa_R
\end{equation*}
for all $\alpha\in(0,1)$ and $\epsilon\in[0,\Tilde\kappa^{-1})$.
\end{lemma}

\begin{proof}
We choose $B_{r_\circ}$, $D_1$, and $D_2$ as in the proof of \cref{L5.1}.
By \cref{EA2.2A,E4.1A}, 
it is evident that $\widetilde{\cI}(\Vo + 3 {\sV}) \in \Lpl^{\infty}(\RR^d)$.
Let $\Hat{x}^\epsilon_\alpha\in\Argmin_{\Bar D_2}\, V^\epsilon_\alpha$.
The function
$\varphi^\epsilon_\alpha\df V^\epsilon_\alpha
- V^\epsilon_\alpha(\Hat{x}^\epsilon_\alpha)$
satisfies
\begin{equation*}
\min_{u\in\Act}\,\bigl[\Ag_u\varphi^\epsilon_\alpha(x) -\alpha\varphi^\epsilon_\alpha(x)
+ \rc^\epsilon(x,u)\bigr]
\,=\, \alpha V^\epsilon_\alpha(\Hat{x}^\epsilon_\alpha)\,\le\,
\frac{\varrho_*}{\mu_{v_*}(B_{r_\circ})}\,,
\end{equation*}
where the inequality follows by \cref{PL5.1A}.
Using \cref{EL5.1B}, we obtain
\begin{equation}\label{PL5.2A}
\sup_{B_R}\,\varphi^\epsilon_\alpha \,\le\, 
\sup_{B_{r_\circ}}\,\varphi^\epsilon_\alpha + \sup_{B_R}\,\bigl(3{\sV} + \Vo\bigr)
\quad\text{for all\ } R>r_\circ\,,
\end{equation}
$\alpha\in(0,1)$ and $\epsilon\in[0,\Tilde\kappa^{-1})$.
Let $v^{\epsilon}_\alpha$ be a measurable selector from the minimizer of
the $\alpha$-discounted problem associated with $\rc^{\epsilon}$.
By the local maximum principle \cite{GilTru}*{Theorem~9.20}, for any $p>0$, there exists
a constant $\Tilde{C}_{1}(p)>0$ such that
\begin{equation*}
\sup_{B_{r_\circ}}\,\varphi^\epsilon_\alpha 
\,\le\, \Tilde{C}_{1}(p)\bigl(\norm{\varphi^\epsilon_\alpha}_{p;D_{1}}
+ \norm{\widetilde\cI\varphi^\epsilon_\alpha}_{\Lp^d(D_1)}
+ \norm{\rc^\epsilon_{v^{\epsilon}_\alpha}}_{\Lp^d(D_1)}\bigr)
\end{equation*}
with $\norm{\varphi^\epsilon_\alpha}_{p;D_{1}} 
\df \bigl(\int_{D_1}\abs{\varphi^{\epsilon}_\alpha(x)}\D{x}\bigr)^{\nicefrac{1}{p}}$,
and by the supersolution estimate \cite{GilTru}*{Theorem~9.22}, there exist some
$p>0$ and $\Tilde{C}_{2}>0$ such that
\begin{equation*}
\norm{\varphi^\epsilon_\alpha}_{p;D_1} \,\le\, \Tilde{C}_{2}\Bigl(
\inf_{D_1}\,\varphi^\epsilon_\alpha+\,\kappa_1\,
\abs{D_2}^{\nicefrac{1}{d}}\Bigr)\,.
\end{equation*}
On the other hand, the inequality in \cref{EL5.1A} implies that
$\inf_{D_1}\,\varphi^\epsilon_\alpha \le \sup_{D_2}\Vo$.
Combining these estimates, we obtain
\begin{equation}\label{PL5.2B}
\sup_{B_{r_\circ}}\,\varphi^\epsilon_\alpha 
\,\le\, \kappa_2 +
\Tilde{C}_{1}(p)\,\norm{\widetilde\cI\,\varphi^\epsilon_\alpha}_{\Lp^d(D_1)}\,.
\end{equation}
where
\begin{equation*}
\kappa_2 \,\df\, \Tilde{C}_{1}(p)\biggl((1+\Tilde{C}_{2})\Bigl(\sup_{D_2}\Vo+\,\kappa_1\,
\abs{D_2}^{\nicefrac{1}{d}}\Bigr) 
+ \norm{\rc^\epsilon_{v^{\epsilon}_\alpha}}_{\Lp^d(D_1)}\biggr)\,.
\end{equation*}
By \cref{PL5.2A,PL5.2B}, we have
\begin{align*}
\sup_{D_2}\, \varphi^\epsilon_\alpha 
\,\le\,\kappa_2 + \norm{\Vo + 3\sV}_{\Lp^{\infty}(D_2)} + 
\Tilde{C}_{1}(p)\,\norm{\widetilde\cI\varphi^\epsilon_\alpha}_{\Lp^d(D_1)}\,.
\end{align*}
Hence, either $\sup_{D_2}\, \varphi^\epsilon_\alpha \le 2\kappa_2 + 
2\norm{\Vo + 3\sV}_{\Lp^{\infty}(D_2)}$,
which directly implies \cref{PL5.1A}, or
\begin{equation}\label{PL5.2C}
\sup_{D_2}\,\varphi^\epsilon_\alpha \,\le\,
2\Tilde{C}_{1}(p)\,\norm{\widetilde\cI\varphi^\epsilon_\alpha}_{\Lp^d(D_1)}\,.
\end{equation}
Suppose that \cref{PL5.2C} is the case.
By \cref{PL5.2A}, we have the estimate
\begin{equation}\label{PL5.2D}
\widetilde{\cI}(\Ind_{D_2^c}\varphi^{\epsilon}_{\alpha})(x) \,\le\,
\biggl(\sup_{B_{r_\circ}}\,\varphi^{\epsilon}_{\alpha}\biggr)\nu(D^c_2)
+ \widetilde{\cI}\bigl(\Ind_{D^c_2} (\Vo + 3\sV)\bigr)(x) \qquad \forall x\in D_1\,.
\end{equation}
Thus, by \cref{PL5.2B,PL5.2C,PL5.2D}, we obtain
\begin{equation*}
\sup_{D_1}\,\widetilde{\cI}\varphi^{\epsilon}_{\alpha} \,\le\,
\kappa_2 \bm\nu +
3\Tilde{C}_1(p)\bm\nu\,\norm{\widetilde\cI\varphi^\epsilon_\alpha}_{\Lp^d(D_1)}
+ \norm{\widetilde{\cI}\bigl(\Ind_{D^c_2} (\Vo + 3\sV)\bigr)}_{\Lp^{\infty}(D_1)}\,.
\end{equation*}
Again we distinguish two cases.
If
\begin{equation*}
\sup_{D_1}\,\widetilde{\cI}\varphi^{\epsilon}_{\alpha} \,\le\,
6\Tilde{C}_1(p)\bm\nu\,
\norm{\widetilde\cI\varphi^\epsilon_\alpha}_{\Lp^d(D_1)}\,,
\end{equation*}
then the proof is the same as in \cite{ACPZ19}*{Theorem~3.3}.
It remains to consider the case
\begin{equation}\label{PL5.2E}
\sup_{D_1}\,\widetilde{\cI}\varphi^{\epsilon}_{\alpha}  \,\le\, 2\kappa_2 \bm\nu
+ 2\norm{\widetilde{\cI}\bigl(\Ind_{D^c_2} (\Vo + 3\sV)\bigr)}_{\Lp^{\infty}(D_1)}\,.
\end{equation}
Let $\Tilde{\phi}_\alpha^\epsilon$ be the solution of the Dirichlet problem
\begin{equation*}
\widetilde\Lg_{v^{\epsilon}_{\alpha}} \Tilde{\phi}^\epsilon_\alpha - 
\alpha\Tilde{\phi}^\epsilon_\alpha  \,=\, 0 \qquad \text{in }D_1
\quad \text{and} \quad \Tilde{\phi}^\epsilon_\alpha 
\,=\, \varphi^{\epsilon}_\alpha \quad \text{on }\partial D_1\,. 
\end{equation*}
By Harnack's inequality, we have $\Tilde\phi^\epsilon_\alpha \le 
\Tilde{C}_H\Tilde{\phi}^\epsilon_{\alpha}(\Hat{x}^\epsilon_\alpha)$ 
for all $x\in B_{r_\circ}$, $\alpha\in(0,1)$, and $\epsilon\in[0,\Tilde\kappa^{-1})$.
Thus
\begin{align*}
\widetilde\Lg_{v^{\epsilon}_\alpha}
(\varphi^\epsilon_\alpha - \Tilde{\phi}^{\epsilon}_\alpha)
- \alpha(\varphi^\epsilon_\alpha - \Tilde{\phi}^{\epsilon}_\alpha)
&\,=\, -\widetilde{\cI}\varphi^\epsilon_\alpha 
+ \alpha V^\epsilon_\alpha(\Hat{x}^\epsilon_\alpha) - \rc^{\epsilon} \\
&\,\ge\, -\sup_{D_1}\,\widetilde{\cI}\varphi^\epsilon_\alpha 
+ \alpha V^\epsilon_\alpha(\Hat{x}^\epsilon_\alpha) - \rc^{\epsilon}
\qquad \text{in }D_1\,,
\end{align*}
and $\varphi^\epsilon_\alpha - \Tilde{\phi}^{\epsilon}_\alpha = 0$ on $\partial D_1$.
On the other hand, we have
\begin{equation}\label{PL5.2F}
\begin{aligned}
\widetilde\Lg_{v^{\epsilon}_\alpha} (\Tilde{\phi}^{\epsilon}_\alpha
- \varphi^\epsilon_\alpha)
- \alpha(\Tilde{\phi}^{\epsilon}_\alpha - \varphi^\epsilon_\alpha)
&\,=\, \widetilde{\cI}\varphi^\epsilon_\alpha 
- \alpha V^\epsilon_\alpha(\Hat{x}^\epsilon_\alpha) + \rc^{\epsilon} \\
&\,\ge\, \inf_{D_1}\,\widetilde{\cI}\varphi^\epsilon_\alpha 
- \alpha V^\epsilon_\alpha(\Hat{x}^\epsilon_\alpha) + \rc^{\epsilon}
\qquad \text{in }D_1\,,
\end{aligned}
\end{equation}
Using \cref{EL5.1A}, we obtain 
\begin{equation}\label{PL5.2G}
\inf_{D_1}\,\widetilde{\cI}\varphi^\epsilon_\alpha \,\ge\,
- \sup_{D_1}\,\widetilde{\cI}\Vo\,.
\end{equation}
Since $\widetilde\cI\Vo \in \Lpl^{\infty}(\RR^d)$,
applying the ABP weak maximum principle in \cite{GilTru}*{Theorem~9.1}
to \cref{PL5.2F}, 
and using \cref{PL5.2E,PL5.2G}, we obtain
$\norm{\varphi_\alpha^\epsilon - \Tilde{\phi}^\epsilon_\alpha}_{\Lp^\infty(D_1)} 
\le \Tilde{C}_0$ for some constant $\Tilde{C}_0$ which does not
depend on $\alpha\in(0,1)$ and $\epsilon\in[0,\Tilde\kappa^{-1})$.
Thus, employing \cite{AA-Harnack}*{Corollary~2.2}
as done in \cite{ACPZ19}*{Theorem~3.3},
we establish \cref{PL5.1A}. This completes the proof.
\end{proof}

In \cref{T5.2} which follows, we derive 
the HJB equation for the ergodic control problem,
and the corresponding characterization of optimal Markov controls.
Compared to \cite[Theorem~4.5]{ACPZ19},
the important difference here is that the solutions to the HJB equation may not be
bounded from below in $\RR^d$, since the running cost function is not near-monotone. 
As a consequence, \cite[Lemma~3.6.9]{ABG12} cannot be applied here directly to establish
the stochastic representation of the solutions, and prove uniqueness.
We let $V_{\alpha} \df\ V_{\alpha}^\epsilon\bigr|_{\epsilon=0}$.

\begin{theorem}\label{T5.2}
Grant \cref{A2.1,A2.2}. Then
\begin{enumerate}
\item[\ttup a]
As $\alpha\searrow 0$,  
$\widetilde{V}_{\alpha} \df V_\alpha - V_\alpha(0)$ converges 
in $\Cc^{1,\rho}$ with $\rho\in (0,1)$, 
uniformly on compact sets, to a function $V_*\in\Sobl^{2,p}(\RR^d)$ for
any $p>1$,
which satisfies $V^-_*\in\sorder(\Vo)$, and
\begin{equation}\label{ET5.2A} 
\min_{u\in\Act}\, \bigl[\Ag_u V_*(x) + \rc(x,u)\bigr] \,=\, \varrho_* 
\quad\text{a.e.\ in\ }\Rd\,.
\end{equation}

\item[\ttup b]
A control $v\in\Ussm$ is optimal for the ergodic control problem with $\rc$ 
if and only if
it is an a.e.\ measurable selector from the minimizer in \cref{ET5.2A}.

\item[\ttup c]
Let $\Ussmb \df \{v\in\Ussm \colon \varrho_v < \infty \}$.
The function $V_*$ is the unique solution (up to an additive constant)
to the equation
$\min_{u\in\Act}\,[\Ag_u V_*(x) + \rc(x,u)] = \varrho$ a.e.\ on $\Rd$,
with $\varrho\le \varrho_*$,
which satisfies
$V^-_*\in\sorder({\Vo})$ and $V_*(0)=0$.
In addition, it
has the stochastic representation
\begin{equation}\label{ET5.2B}
V_*(x) \,=\, \lim_{r\searrow 0}\,
\inf_{v\in\Ussmb}\, \Exp_x^{v} \biggl[
\int_0^{\uuptau_r}\bigl(\rc_{v}(X_s) - \varrho_* \bigr)\D{s} \biggr]\,.
\end{equation}
\end{enumerate}
\end{theorem}

\begin{proof}
We first prove (a).
By \cref{ET5.1A}, we have 
\begin{equation}\label{PT5.2A}
\min_{u\in\Act}\, \bigl[\Ag_u \widetilde{V}_{\alpha}(x) + \rc(x,u)\bigr] \,=\, 
\alpha\widetilde{V}_{\alpha}(x) + \alpha V_{\alpha}(0) \quad\text{a.e.\ in\ }\Rd\,. 
\end{equation}
The limit
$\lim_{\alpha \searrow 0}\alpha V_{\alpha}(0) = \varrho_*$
follows  as in the proof of Theorem~3.6 in \cite{ABP15}
using \cref{L5.2}.

We fix an arbitrary ball $\sB$, and using \cref{EL5.1A,EL5.1B}, we obtain
\begin{equation}\label{PT5.2B}
\abs{\widetilde{V}_{\alpha}(x)} \,\le\, 
\norm{\widetilde{V}_\alpha}_{L^\infty(\sB)}+ \Vo(x) + 3\sV(x) 
\qquad \forall x\in\RR^d\,.
\end{equation}
By \cref{PT5.2B,E4.1A,EA2.2A}, 
we have $\widetilde{\cI}\abs{\widetilde{V}_\alpha}\in\Lpl^{\infty}(\RR^d)$.
Let $v_\alpha$ be a measurable selector for the minimizer of \cref{PT5.2A}.
Then, applying the interior estimate in \cite{GilTru}*{Theorem~9.11}, 
we obtain
\begin{equation*}
\norm{\widetilde{V}_{\alpha}}_{\Sob^{2,p}(B_R)} \,\le\,
C\Bigl( \norm{\widetilde{V}_{\alpha}}_{L^p(B_{2R})}  
+ \norm{\alpha V_\alpha(0)  -  \rc_{v_\alpha} 
- \widetilde{\cI}\widetilde{V}_{\alpha}}_{L^p(B_{2R})}\Bigr)\,, 
\end{equation*}
where $C \equiv C(R,p)$. 
Hence, $\sup_{\alpha\in(0,1)}\norm{\widetilde{V}_{\alpha}}_{\Sob^{2,p}(B_R)} < \infty$.
Thus, following a standard argument (see \cite{ABG12}*{Lemma~3.5.4}), 
for any sequence $\alpha_n\searrow 0$, the functions 
$\{\widetilde{V}_{\alpha_n}\}$ converge
along a subsequence in $\Cc^{1,\rho}$ with $\rho\in(0,1)$, 
uniformly on compact sets, to $V_*$.
Since $R$ is arbitrary, 
this proves \cref{ET5.2A}.
Letting $\alpha\searrow 0$ 
in \cref{EL5.1A}, which holds for all $\delta>0$, with $\Tilde{r}$ depending
only on $\delta$, we obtain (with $\epsilon = 0$)
\begin{equation}\label{PT5.2C}
V^-_* \,\in\, \sorder(\Vo)\,.
\end{equation}

Concerning part (b), necessity follows by \cite{ACPZ19}*{Theorem~3.5}.
Sufficiency follows exactly as in the proof of
\cite{ABP15}*{Theorem~3.4\,(b)} using \cref{PT5.2C} and part (a).

It remains to establish uniqueness and the stochastic representation
as stated in part (c).
By \cref{PT5.2B}, we also have $\widetilde{\cI}V_*\in\Lpl^{\infty}(\RR^d)$.   
Following the same arguments as in \cite{ABG12}*{Lemma~3.6.9},
we have
\begin{equation}\label{PT5.2D}
V_*(x) \,\le\, \liminf_{r\searrow0}\,\inf_{v\in\Ussmb}
\Exp_x^v\biggl[\int_0^{\uuptau_r}(\rc_v(X_s) - \varrho_*)\,\D{s}\biggr]\,.
\end{equation}
On the other hand, applying It\^{o}'s formula, we obtain
\begin{equation}\label{PT5.2E}
V_*(x) \,=\, \Exp^{v_0}_x\biggl[ \int_0^{\uuptau_r\wedge\uptau_R}
\bigl(\rc_{v_0}(X_s) - \varrho_* \bigr)\,\D{s}
+ V_*(X_{\uuptau_r\wedge\uptau_R})\biggr]\,,
\end{equation}
with $v_0$ a measurable selector for the minimizer of \cref{ET5.2A}.
Note that
\begin{equation*}
V_*(X_{\uuptau_r\wedge\uptau_R}) \,=\,
V_*(X_{\uuptau_r})\Ind(\uuptau_r<\uptau_R)
+ V_*(X_{\uptau_R})\Ind(\uuptau_r\ge\uptau_R)\,.
\end{equation*}
We next show that 
\begin{equation}\label{PT5.2F}
\limsup_{R\nearrow\infty}\,\Exp^{v_0}_x
[ V^{-}_*(X_{\uptau_R})\Ind(\uuptau_r\ge\uptau_R)]\,=\,0\,.
\end{equation}
Let $\Phi\df V_* + \Vo$.
It follows by \cref{PT5.2C} that
$\Phi$ is coercive and $V^{-}_*\in\sorder(\Phi)$.
It is also evident by \cref{E4.1A,ET5.2A} that
$\Ag_{v_0}\Phi(x) \le \Ind_{\sB_\circ}(x) + \varrho_*$.
Then, applying It\^{o}'s formula, we obtain
\begin{equation*}
\Exp_x^{v_0}\bigl[\Phi(X_{\uuptau_r\wedge\uptau_R})\bigr]
\,\le\,  (\varrho_*+1) \Exp_x^{v_0}[\uuptau_r\wedge\uptau_R] + \Phi(x)
\,\le\, (\varrho_*+1) \Exp_x^{v_0}[\uuptau_r] + \Phi(x)\,.
\end{equation*}
We also have 
\begin{equation}\label{PT5.2G}
\Exp_x^{v_0}\bigl[V^{-}_*(X_{\uptau_R})\Ind(\uuptau_r\ge\uptau_R)\bigr]
\,\le\, \Bigl((\varrho_*+1) \Exp_x^{v_0}[\uuptau_r] + \Phi(x)\Bigr)\,
\sup_{y\in B_R^c}\, \frac{V^{-}_*(y)}{\Phi(y)}\,.
\end{equation}
Note that $\Exp_x^{v_0}[\uuptau_r]$ is finite since $v_0\in\Ussm$ by \cref{T4.1}.
Since $V_*^-\in\sorder(\Phi)$, 
it follows that
 $\sup_{y\in B_R^c}\, \frac{V^{-}_*(y)}{\Phi(y)}$ vanishes as $R\to\infty$, 
which in turn implies that
the right hand side\ of \cref{PT5.2G} converges to $0$ as $R\to\infty$.
This proves \cref{PT5.2F}.
Letting $R\to\infty$ in \cref{PT5.2E},
it follows by Fatou's lemma and \cref{PT5.2F} that
\begin{equation*}
V_*(x) \,\ge\, \Exp^{v_0}_x\biggl[ \int_0^{\uuptau_r}
\bigl(\rc_{v_0}(X_s) - \varrho_* \bigr)\,\D{s}
+ V_*(X_{\uuptau_r})\biggr]\,,
\end{equation*}
and we obtain
\begin{equation}\label{PT5.2H}
V_*(x) \,\ge\, \limsup_{r\searrow0}\,\inf_{v\in\Ussmb}
\Exp_x^v\biggl[\int_0^{\uuptau_r}(\rc_v(X_s) - \varrho_*)\,\D{s}\biggr]\,,
\end{equation}
which together with \cref{PT5.2D} implies \cref{ET5.2B}.
Note that, by the argument above, \cref{PT5.2H} holds for any solution $V$
of \cref{ET5.2A} which satisfies $V^-\in\sorder({\Vo})$.
Thus, if $V$ is any other solution, we have $V_*\le V$ and equality
follows by the strong maximum principle.
This completes the proof.
\end{proof}

\subsection{Regularity of solutions of the HJB}

In this section, we examine the regularity of solutions of the HJB equations
in \cref{T5.1,T5.2}.
If the L\'{e}vy measure $\nu$ has a compact support,
then it follows by the elliptic regularity \cite[Theorem~9.19]{GilTru} that
the solutions to the HJB equations are in $\Cc^{2,r}(\RR^d)$ for any $r\in(0,1)$.
See also Remark~3.4 in \cite{ACPZ19}.

We need the following gradient estimate which is also applicable
to a larger class of equations.

\begin{lemma}\label{L5.3}
Let $\varphi \in\Sobl^{2,p}(\RR^d)$, $p>d$, be a strong solution, 
having at most polynomial growth of degree $m>0$, to the equation
\begin{equation}\label{EL5.3A} 
a^{ij}(x)\partial_{ij}\varphi(x) + b^{i}(x)\partial_i\varphi(x) + c(x)\varphi(x)
+ \widetilde{\cI}\varphi(x) \,=\, f(x) \quad \text{on } \RR^d \,, 
\end{equation}
where
\begin{enumerate}
\item[\ttup i]
the matrix $a$ is bounded, Lipschitz continuous
on $\RR^d$ and uniformly elliptic;
\item[\ttup{ii}]
the coefficients $b$ and $c$ are locally bounded and
measurable, with $b$ having at most linear growth and $c$ having at most quadratic growth;
\item[\ttup{iii}]
the function $f$ has at most polynomial growth of degree $\kappa$ with 
$\kappa\in(0,m+2]$;
\item[\ttup{iv}]
the L\'evy measure $\nu$ of the nonlocal operator $\widetilde{\cI}$ 
is finite and satisfies $\int_{\RR^d}\abs{z}^m\,\nu(\D{z}) < \infty$.
\end{enumerate}
Then, $\abs{\grad{\varphi}(x)} \in \order(\abs{x}^{m+1})$.
\end{lemma}

\begin{proof}
For any fixed $x_0\in\RR^d$, for which without loss of generality we assume
$\abs{x_0} \ge 1$,
we define the scaled variables
\begin{equation*}
\Tilde{\varphi}(x) \,\df\, \varphi\biggl(\frac{x}{\abs{x_0}^{\nicefrac{1}{2}}} \biggr)\,,
\end{equation*}
and similarly for $\Tilde{a}$, $\Tilde{b}$, $\Tilde{c}$, and $\Tilde{f}$.
The equation in \cref{EL5.3A} then takes the form
\begin{equation}\label{PL5.3A}
\Tilde{a}^{ij}(x) \partial_{ij}\Tilde{\varphi}(x)
+  \frac{\Tilde{b}^i(x)}{\abs{x_0}^{\nicefrac{1}{2}}}\, \partial_{i}\Tilde\varphi(x)
+ \frac{\Tilde{c}(x)}{\abs{x_0}}\Tilde\varphi(x) 
+  \widetilde{\cI}\varphi\biggl(\frac{x}{\abs{x_0}^{\nicefrac{1}{2}}}\biggr) 
\,=\,  \frac{\Tilde f(x)}{\abs{x_0}} \quad \text{on }\RR^d\,.
\end{equation} 
It is clear from (i)--(ii) that the coefficients $\Tilde{a}^{ij}$,
 $\abs{x_0}^{\nicefrac{-1}{2}}\Tilde{b}^i$ and $\abs{x_0}^{-1}\Tilde{c}$
are bounded in the ball $B_2(x_0)$, with a bound independent of $x_0$, and  that
the Lipschitz and ellipticity constants of the matrix $\Tilde{a}$
in $B_2(x_0)$ are independent of $x_0$.
Thus, it follows by \cref{PL5.3A} and the a priori estimate in
\cite{GilTru}*{Theorem~9.11} that, for any fixed $p>d$, we have
\begin{equation}\label{PL5.3B}
\norm{\Tilde\varphi}_{\Sob^{2,p}(B_1(x_0))}\,\le\,
C\Bigl(\bnorm{\Tilde\varphi}_{\Lp^{p}(B_2(x_0))} 
+ \bnorm{\widetilde{\cI}\varphi(\abs{x_0}^{\nicefrac{-1}{2}}\,\cdot\,)}_{\Lp^{p}(B_2(x_0))}
+ \abs{x_0}^{-1}\bnorm{\Tilde f}_{\Lp^{p}(B_2(x_0))}\Bigr)
\end{equation}
for some positive constant $C$ independent of $x_0$.
Since $\nu$ is finite and $\int_{\RR^d}\abs{z}^m\,\nu(\D{z}) < \infty$,
it follows that $\widetilde{\cI}\varphi$ has at most polynomial growth of degree $m$. 
Then, by the assumptions of the lemma, the right hand side of \cref{PL5.3B} is 
$\order\bigl(\abs{x_0}^{\nicefrac{m}{2}}\bigr)$. 
Therefore, by \cref{PL5.3B} and the compactness of the
Sobolev embedding
$\Sob^{2,p}\bigl(B_1(x_0)\bigr) \hookrightarrow \Cc^{1,r}\bigl(B_1(x_0)\bigr)$,
for $0 < r < 1 -\frac{d}{p}$, we obtain
\begin{equation}\label{PL5.3C}
\bnorm{\grad \tilde{\varphi}}_{\Lp^{\infty}(B_1(x_0))}\,\le\, 
C_0\bigl(1 + \abs{x_0}^{\nicefrac{m}{2}}\bigr) \qquad \forall\,x_0\in\RR^d
\end{equation}
for some positive constant $C_0$ independent of $x_0$.
On the other hand, 
\begin{equation*}
\grad{\varphi}\biggl(\frac{x_0}{\abs{x_0}^{\nicefrac{1}{2}}}\biggr) \,=\,
\abs{x_0}^{\nicefrac{1}{2}}\grad\Tilde\varphi(x_0) \qquad \forall\,x_0\in\RR^d\,,
\end{equation*}
which together with \cref{PL5.3C} imply
$\abs{\grad\varphi(x)} \in \order(\abs{x}^{m+1})$.
This completes the proof.
\end{proof}

Consider the following assumption on the growth of the coefficients
and the functions $\Vo$ and $\sV$.

\begin{assumption}\label{A5.1}
\begin{enumerate}
\item[\ttup i]
The running cost function   
has at most polynomial growth of degree $m_\circ\ge 1$, that is,
$\rc(x,u) \le C_\circ (1 + \abs{x}^{m_\circ})$, 
for all $(x,u)\in\RR^d\times\Act$
and some positive constant $C_\circ$.

\item[\ttup{ii}]
\cref{A2.1,A2.2} hold with $\sV$ and $\Vo$ having at most polynomial
growth of degree $m_\circ$.

\item[\ttup{iii}]
The L\'evy measure $\nu$ satisfies
$\int_{\RR^d}\abs{z}^{m_\circ+ 1}\,\D{z} < \infty$.
\end{enumerate}
\end{assumption}

\begin{remark}
Provided that $\int_{\RR^d}\abs{z}^{m+1}\,\nu(\D{z})<\infty$,
it is clear that \cref{A5.1} holds for the limiting controlled diffusion
in \cref{S3}.
\end{remark}

We have the following theorem.

\begin{theorem}\label{T5.3}
Grant \cref{A5.1}. The solutions $V^\epsilon_\alpha$ 
of \cref{ET5.1A}, and $V_*$ of \cref{ET5.2A}  are in $\Cc^{2,r}(\RR^d)$
for any $r\in (0,1)$. 
\end{theorem}

\begin{proof}
Consider $V_*$. 
Since $V_*\in\order(\Vo + 3\sV)$ by \cref{PT5.2B},
then $\widetilde{\cI}V_*\in\Lpl^{\infty}(\RR^d)$ by \cref{EA2.1A,EA2.2A},  and 
$V_*$ has at most polynomial growth of degree $m_\circ$ by \cref{A5.1}.
We claim that $\widetilde{\cI}V_*$ is locally Lipschitz continuous.
To prove the claim, we fix some $x_0\in\Rd$, and write
\begin{equation}\label{PT5.3B}
\begin{aligned}
\babs{\widetilde\cI V_*(x_0 + x^{\prime}) - \widetilde{\cI} V_*(x_0 + x^{\prime\prime})} 
&\,\le\,\int_{\RR^d}\int_0^{1} 
 \babs{\bigl\langle\grad V_*(x_0 + \theta(x^{\prime} - x^{\prime\prime}) + y), 
x^{\prime} - x^{\prime\prime}\bigr\rangle} \,\D{\theta}\,\nu(\D y) \\
&\,\le\, \babs{x^{\prime} - x^{\prime\prime}}
\int_{\RR^d}\norm{\grad V_{*}}_{\Lp^{\infty}(B_2(x_0+y))} \,\nu(\D{y})
\end{aligned}
\end{equation}
for all $x',x''\in B_1$.
By \cref{ET5.2A} and \cref{L5.3} we obtain
\begin{equation*}
\norm{\grad V_{*}}_{\Lp^{\infty}(B_2(z))} \,\in\,
\order(\abs{z}^{m_\circ+1})\,,
\end{equation*}
which together with \cref{A5.1}\,(iii) and \cref{PT5.3B}
proves the claim.
It then follows by \cref{ET5.2A} and elliptic regularity
(see \cite{GilTru}*{Theorem~9.19}) that
$V_*$ is in $\Cc^{2,r}(\RR^d)$ for any $r\in (0,1)$.
The proof of the same property for $V^\epsilon_\alpha$ is completely analogous.
\end{proof}

\section{Pathwise optimality}\label{S6}

The pathwise formulation of the ergodic control problem
seeks to a.s. minimize over $U\in\Uadm$ 
\begin{equation*}
\limsup_{t\rightarrow\infty} \frac{1}{t}\int_0^{t}\rc(X_s^U,U_s)\,\D{s}\,,
\end{equation*}
where $\process{X^U}$ denotes the process governed by \cref{E-sde} under the control
$U$.
If the running cost is near-monotone
or a uniform stability condition holds,
it follows by \cite[Theorem~4.4]{ACPZ19} that  
every average cost optimal stationary Markov control is also optimal 
with respect to the pathwise ergodic criterion.
In this section, we extend the results of the diffusion model in \cite{AA17} to
the jump-diffusion model.
We modify \cref{A2.1} as follows.

\begin{assumption}\label{A6.1}
\Cref{A2.1} holds with $F=\phi\comp\Vo$, where
$\phi\colon\RR_+\to\RR_+$ is smooth increasing concave function,
satisfying $\phi(z)\to\infty$ as $z\to\infty$.
In addition,
the functions $\upsigma$ and $\frac{\grad\Vo}{1+\phi\comp\Vo}$
are bounded on $\Rd$.
\end{assumption}

\begin{remark}
For the examples in \cref{S3}, 
we may rescale $F(x)$ in \cref{L3.1} 
so that $F(x) = C \sV_{Q,m}(x)$ for some positive constant $C$.
Thus, if we choose $\phi(x) = C x$, then \cref{A6.1} holds. 
\end{remark}

\begin{theorem}
Grant \cref{A6.1}.
Then, if $U\in\Uadm$ is such that $\varrho_U(x)<\infty$,
the family of random empirical measures
$\{\zeta^U_{t}\,\colon t>0\}$ in \cref{D4.1} is tight a.s.
In particular,
every average cost optimal stationary Markov control is
pathwise optimal.
\end{theorem}

\begin{proof}
The technique is similar to that of \cite{AA17}*{Theorem~3.1},
but here we need to account for the nonlocal term.
We define $\psi_{N}\colon\RR_+\to\RR_+$,
for $N\in\NN$, by
\begin{equation*}
\psi_{N}(z) \,\df\, \int_{0}^{z} \frac{1}{N+\phi(y)}\,\D{y}\,,
\end{equation*}
and $\varphi_{N}\colon\Rd\to\RR_+$ by
$\varphi_{N}\df\psi_{N}\comp\Vo$, where `$\comp$'
denotes composition of functions.
Let $U\in\Uadm$ be some admissible control such that
\begin{equation*}
\limsup_{n\to\infty}\,\int_{\Rd\times\Act} \rc(x,z)\,\zeta^U_{t_n}(\D{x},\D{z})
\,<\,\infty\,,
\end{equation*}
for some increasing divergent sequence $\{t_n\}$.
Since $\rc$ is coercive on $\cK\times\Act$,
it follows that $\zeta^U_{t_n}$ is a.s.\ tight
when restricted to $\Bor(\cK\times\Act)$, the Borel $\sigma$-algebra
of $\cK\times\Act$.

Since $\grad\varphi_N$ and $\upsigma$ are bounded,
by It\^o's formula
and \cref{EA2.1A} we obtain
\begin{equation}\label{PT6.1A}
\begin{aligned}
\frac{\varphi_N(X_t) - \varphi_N(X_0)}{t}
&\,=\, \frac{1}{t}\,\int_0^t \Ag_{U_s} \varphi_N(X_s)\,\D{s} +
\frac{1}{t}\,\int_0^t \langle \grad\varphi_N(X_s),\upsigma(X_s)\, \D W_s\rangle   \\
&\mspace{30mu} +
\frac{1}{t}\,\int_{0}^{t}\int_{R^m_*}
\Bigl(\varphi_N\bigl(X_{s-} + g(\xi)\bigr)- \varphi_N(X_{s-})\Bigr)
\widetilde\cN(\D s, \D \xi)\,.
\end{aligned}
\end{equation}
Then, the second and third terms on the right hand side of \cref{PT6.1A}
converge to $0$ a.s.\ as $t\to\infty$,
by the proof of Lemma~4.2 in \cite{ACPZ19}.

An easy computation shows that
\begin{equation}\label{PT6.1B}
\Ag_{z} \varphi_N(x) \,\le\, \begin{cases}
\frac{1-\phi(\Vo(x))}{N+\phi(\Vo(x))}
+ F_{1,N}(x)+F_{2,N}(x)\qquad\text{on~}\cK^c\,,\\[5pt]
\frac{1 + \rc(x,z)}{N+\phi(\Vo(x))}
+ F_{1,N}(x)+F_{2,N}(x)\qquad\text{on~}\cK\,,
\end{cases}
\end{equation}
where
\begin{equation}\label{PT6.1C}
F_{1,N}(x)\,\df\,-\phi'\bigl(\Vo(x)\bigr)\,
\frac{\babs{\upsigma\transp(x)\nabla\Vo(x)}^2}{\bigl(N+\phi(\Vo(x))\bigr)^2}\,,
\qquad x\in\Rd\,,
\end{equation}
and
\begin{equation}\label{PT6.1D}
F_{2,N}(x) \,\df\, 
\frac{1}{N+\phi\bigl(\Vo(x)\bigr)}
\int_{\Rd} \int_{\Vo(x)}^{\Vo(x+\xi)}\frac{\phi\bigl(\Vo(x)\bigr)-\phi(y)}{N+\phi(y)}\,
\D y\,\nu(\D\xi)\,.
\end{equation}
To estimate \cref{PT6.1D}, we use
\begin{equation*}
\babss{\int_{\Rd} \int_{\Vo(x)}^{\Vo(x+\xi)}
\frac{\phi\bigl(\Vo(x)\bigr)-\phi(y)}{N+\phi(y)}\,\D y\,\nu(\D\xi)}
\,\le\, \nu(\Rd)\,\norm{\phi'}_{\Lp^\infty(\RR)}\,
\bnormm{\frac{\grad\Vo}{N+ \varphi\comp\Vo}}_{\Lp^\infty(\Rd)}\,.
\end{equation*}
Then combining this with \cref{PT6.1A,PT6.1B,PT6.1C,PT6.1D},
we obtain
\begin{equation*}
\limsup_{n\to\infty}\,
\int_{\cK^c\times\Act} \frac{\phi\bigl(\Vo(x)\bigr)}
{N+\phi\bigl(\Vo(x)\bigr)}\,\zeta^U_{t_n}(\D{x},\D{z})
\,\le\, \frac{C}{N}
\end{equation*}
for some constant $C$,
from which it follows that $\zeta^U_{t_n}$ is a.s.\ tight
when restricted to $\Bor(\cK^c\times\Act)$.
Therefore, it is a.s.\ tight in  $\cP(\Rd\times\Act)$.
This completes the proof.
\end{proof}

\section{An approximate HJB equation}\label{S7}

In this section, we use an approximate HJB equation to 
construct $\epsilon$-optimal controls.
Its purpose is twofold.
First, it is used to establish asymptotic optimality in \cite{APZ19b}.
Second, the approximating HJB equation is a semilinear equation on
a sufficiently large ball, and a linear equation on its complement, which
is beneficial to numerical methods.
This result was first reported in \cite[Section~4]{ABP15} for a continuous diffusion.
The proof in that paper crucially relied on the the following property of a
positive recurrent nondegenerate diffusion of the form
$\D X_t = b(X_t)\D t + \upsigma(X_t)\D W_t$:
If a function $f\colon\Rd\to\RR$ is integrable under the invariant probability
measure $\mu$ of the diffusion, then $\Exp[f(X_t)]\to \mu(f)$ as $t\to\infty$.
The proof of this property in \cite[Proposition~2.6]{Ichihara-13} relies on
the Harnack property which we do not have for the model in this paper.
Thus, a different approach is adopted.

We first consider the ergodic control problem with a suitable control
which satisfies \cref{A2.2} fixed outside a ball of arbitrarily large radius.
Then, we show that $\epsilon$-optimal controls are obtained by choosing the radius of the
ball sufficiently large. 
 \Cref{A7.1new} replaces \cref{A2.2} in this section.

\begin{assumption}\label{A7.1new}
The following hold:
\begin{enumerate}
\item[\ttup{i}]
The function
$\widetilde{F}$ in \cref{E-tildF} is in
$\Cc^2(\Rd)$, has at most polynomial growth of degree $\widetilde{m} \ge 1$,
and satisfies
\begin{equation*}
\abs{x}\abs{\grad \widetilde{F}}
+ \abs{x}^2\norm{\grad^2 \widetilde{F}}\,\in\,\order(\widetilde{F})\,.
\end{equation*}
\item[\ttup{ii}]
There exist $\Tilde{v}\in\Ussm$ and $\widetilde\sV\in\Cc^2(\Rd)$, with
$\widetilde\sV\in\order(\widetilde{F})$, satisfying
\begin{equation*}
\Ag_{\Tilde{v}} \widetilde\sV \,\le\, \widetilde{C} - \widetilde{F}
\end{equation*}
for some positive constant $\widetilde{C}$;
\item[\ttup{iii}]
The L\'evy measure $\nu$ satisfies
$\int_{\Rd}\abs{z}^{\widetilde{m}}\,\nu(\D{z}) < \infty$.
\end{enumerate}
\end{assumption}

The jump diffusion with a `truncated' control space is defined as following.
\begin{definition}
With $\Tilde{v}\in\Ussm$ as in \cref{A7.1new} and each $R>0$, we define
\begin{align*}
b^{R}(x,u) &\,\df\, 
\begin{cases}b(x,u)\,, \qquad &\text{if\ } (x,u)\in B_R\times\Act\,, \\
b\bigl(x,\Tilde{v}(x)\bigr)\,, \qquad &\text{if\ } x\in B^c_R\,,
\end{cases}  \\
\rc^{R}(x,u) &\,\df\, 
\begin{cases}\rc(x,u)\,, \qquad &\text{if\ } (x,u)\in B_R\times\Act\,, \\
\rc\bigl(x,\Tilde{v}(x)\bigr)\,, \qquad &\text{if\ } x\in B^c_R\,.
\end{cases} 
\end{align*}
Let $\Ag^R_u$ denote the operator associated with the controlled jump diffusion
\begin{equation*}
\D{X_t} \,\df\,
b^R(X_t,U_t)\,\D{t} + \upsigma(X_t)\,\D{W_t} + \D{L_t}\,,
\end{equation*} 
with $X_0 = x\in\RR^d$,
and define
\begin{equation*}
\rc^{\epsilon,R}\,\df\, \rc^R + \epsilon\widetilde{F}\,,
\quad\text{and}\quad
\varrho^{\epsilon,R}_{*} \,=\, \inf_{v\in\Usm(\Tilde{v},R)} \uppi_v(\rc^{\epsilon})\,,
\end{equation*}
where $\Usm(\Tilde{v},R)$ denotes the class of stationary Markov
controls which agree with $\Tilde{v}\in\Ussm$ on $B_R^c$.
\end{definition}

By \cref{T5.2}, for each $R>0$ and $\epsilon\in(0,1)$, 
there exists a unique $V^R_\epsilon\in\Sobl^{2,p}(\RR^d)$, for any
$p>1$, which is bounded from below in $\RR^d$ and satisfies $V^R_\epsilon(0)=0$, and
\begin{equation*} 
\min_{u\in\Act}\,\bigl[\Ag^R_u V^R_\epsilon(x) + \rc^{\epsilon,R}(x,u)\bigr] 
\,=\, \varrho^{\epsilon,R}_{*}
\quad\text{a.e.\ in\ }\Rd\,.
\end{equation*}
In addition, there exists a constant
$C$ such that
\begin{equation}\label{ES7.1B}
V^R_\epsilon \,\le\, C(1+2\widetilde{\sV})\qquad\forall\,R>0\,.
\end{equation}
It is clear that $\varrho^{\epsilon,R}_{*}$ is nonincreasing.
Let $\Hat{\varrho}^\epsilon \df \lim_{R\to\infty} \varrho^{\epsilon,R}_{*}$.
As in  the proof of \cref{T5.2},
$V^R_\epsilon \rightarrow \widehat{V}_\epsilon\in\Sobl^{2,p}(\RR^d)$, for any
$p>1$, which satisfies
\begin{equation}\label{ES7.1C}
\min_{u\in\Act}\,\bigl[\Ag_u \widehat{V}_\epsilon(x) + \rc^\epsilon(x,u)\bigr] \,=\,
\Hat{\varrho}^\epsilon 
\quad\text{a.e.\ in\ }\Rd\,.
\end{equation}

Recall that $\rc^\epsilon$, defined in the proof of \cref{T4.1},
is the optimal ergodic value for the controlled diffusion in \cref{E-sde}
with running cost $\rc^\epsilon$.
We wish to show that $\Hat{\varrho}^\epsilon=\varrho^\epsilon_*$.
To establish this we need the following lemma,
which provides a lower bound for supersolutions of a general class
of integro-differential equations.

\begin{lemma}\label{L7.1}
Let $\varphi \in \Sobl^{2,d}(\RR^d)$, be a supersolution of the equation
\begin{equation*}
a^{ij}(x)\partial_{ij}\varphi(x) + b^{i}(x)\partial_i\varphi(x)
+ {\cI}\varphi(x) + \widetilde{F}(x) \,=\, 0 \quad \text{on } \RR^d \,, 
\end{equation*}
which is bounded below in $\Rd$.
Assume the following:
\begin{enumerate}
\item[\ttup{i}] the matrix $a$ is nonsingular and satisfies
\begin{equation*}
\limsup_{\abs{x}\rightarrow\infty} \frac{\norm{a(x)}}{\abs{x}^2} \,<\,\infty\,;
\end{equation*}

\item[\ttup{ii}] 
the drift $b$ is measurable and has at most linear growth;

\item[\ttup{iii}] the function $\widetilde{F}$ is as in
\cref{A7.1new};

\item[\ttup{iv}] 
the operator $\cI\colon \Cc^1(\RR^d)\mapsto\Cc(\RR^d)$ is given by
\begin{equation*}
\cI h(x) \,=\, \int_{\RR^d_*} 
h(x+y) - h(x) - \langle y,\grad{h(x)} \rangle\,\nu(\D{y})
\end{equation*}
 for $h\in\Cc^1(\RR^d)$
and the L\'evy measure $\nu$ satisfies 
\begin{equation*}
\int_{\RR^d_*}\abs{z}^{\widetilde{m}}\,\nu(\D{z}) 
+ \int_{B_1\setminus\{0\}}\abs{z}^{2}\,\nu(\D{z})  < \infty\,.
\end{equation*}
\end{enumerate}
Then, $\widetilde{F}\in\order(\varphi)$.
\end{lemma}

\begin{proof}
We have
\begin{equation}\label{PL7.1A}
\sA \varphi(x)
\,\df\, a^{ij}(x)\partial_{ij}\varphi(x) + b^{i}(x)\partial_i\varphi(x)
+ {\cI}\varphi(x) \,\le\, -\widetilde{F}(x) \qquad \forall\,x\in \RR^d\,.
\end{equation}
By using \ttup{i}--\ttup{iii}, 
it is clear that $\sA \widetilde{F} - {\cI}\widetilde{F}\in \order(\widetilde{F})$.
By \ttup{iii} and \ttup{iv}, 
it follows by \cite{APS19}*{Lemma 5.1} that ${\cI}\widetilde{F} \in \order(\widetilde{F})$.
Thus, there exists $\Hat{r} > 0$ such that 
\begin{equation}\label{PL7.1B}
\abs{\sA \widetilde{F}(x)} \,\le\, C(1 + \widetilde{F}(x)) \qquad \forall\, x\in B^c_{\Hat{r}} \,,  
\end{equation}
for some positive constant $C$.
Let $\phi_n(x)$ be a smooth cutoff function satisfying $\phi_n(x) = 1$ on $B_n$ 
and $\phi_n(x) = 0$ on $B^c_{n+1}$, for $n\in\NN$.
By \cref{PL7.1B} and \ttup{iii},
we can choose $r > \Hat{r}$ large enough and $\epsilon\in(0,1)$ sufficiently small so that
for any $n\in\NN$,
\begin{equation}\label{PL7.1C}
-\widetilde{F}(x) - \epsilon\,\sA \widetilde{F}(x)\phi_n(x) 
\,\le\, 0 \qquad \forall\,x\in B^{c}_r\,.
\end{equation}
Let $M$ be a lower bound for $\varphi$, and $n > r$. 
We define the function 
$\Hat{\varphi}^n_\epsilon(x) \df \varphi(x) - \epsilon \widetilde{F}(x) \phi_n(x) 
- (M - \sup_{B_r} \widetilde{F})$.
Then, applying \cref{PL7.1A,PL7.1C}, we have 
\begin{equation*}
\sA \Hat{\varphi}^n_\epsilon(x) 
\,\le\, -\widetilde{F}(x) - \epsilon\,\sA \widetilde{F}(x)\phi_n(x) \,\le\, 0
\qquad \forall\,x\in B_{n+1}\setminus B_r\,.
\end{equation*}
It is evident that $\varphi(x) - (M - \sup_{B_r} \widetilde{F}) \ge 0$ on $B^c_{n+1}$, and 
$\varphi(x) - \epsilon \widetilde{F}(x) - (M - \sup_{B_r} \widetilde{F}) \ge 0$ on $B_r$.
Thus, applying the strong maximum principle, 
we obtain $\Hat{\varphi}^n_\epsilon(x) \ge 0$ in $\RR^d$.
It follows that $\varphi(x) \ge \epsilon \widetilde{F}(x)\phi_n(x) 
+ (M - \sup_{B_r} \widetilde{F})$ in $\RR^d$ 
for all $n$ large enough. This completes the proof.
\end{proof}

The main result of this section is the following.

\begin{theorem}\label{T7main}
Grant \cref{A7.1new}. Then, $\Hat{\varrho}^\epsilon=\varrho^\epsilon_*$.
\end{theorem}

\begin{proof}
Let $V_\epsilon\in\Sobl^{2,p}(\RR^d)$, $p>1$,
be the unique solution of the
equation
\begin{equation}\label{PT7mainA} 
\min_{u\in\Act}\,\bigl[\Ag_u V_\epsilon(x) + \rc^\epsilon(x,u)\bigr] \,=\,
\varrho^\epsilon_*
\quad\text{a.e.\ in\ }\Rd\,,
\end{equation}
which is bounded below in $\Rd$ and satisfies $V_\epsilon(0)=0$.
Applying It\^o's formula to \cref{PT7mainA}, we obtain
\begin{equation}\label{PT7mainB}
\Exp^{v^\epsilon_*}_x\bigl[V_\epsilon(X_{T\wedge\uptau_r})\bigr] \,=\,
V_{\epsilon}(x) - 
\Exp^{v^\epsilon_*}_x\biggl[\int_0^{T\wedge\uptau_r}\rc^\epsilon_{v^\epsilon_*}
\bigl(X_s\bigr) 
\,\D{s}\biggr]  + \varrho^\epsilon_*\Exp^{v^\epsilon_*}_x[T\wedge\uptau_r]\,,
\end{equation}
with $v^\epsilon_*$ a measurable selector from the minimizer
of \cref{PT7mainB}.
It is clear from \cref{PT7mainB} that
$G(T)\df\lim_{r\to\infty} \Exp^{v^\epsilon_*}_x\bigl[V_{\epsilon}(X_{T\wedge\uptau_r})\bigr]$
exists and satisfies $\limsup\frac{1}{T} G(T)\to0$ by Birkhoff's ergodic theorem.
By \cref{L7.1}, we have
$\rc^\epsilon_{v^\epsilon_*}\in\order(V_{\epsilon})$, and this
implies  that
$\widehat{V}_\epsilon\in\order(V_{\epsilon})$
by \cref{A7.1new}\,(ii) and \cref{ES7.1B}.
Therefore, if
$\widehat{G}(T) \df \limsup_{r\to\infty}\Exp^{v^\epsilon_*}_x\bigl[\widehat{V}_\epsilon(X_{T\wedge\uptau_r})\bigr]$,
then $\limsup_{T\to\infty}\frac{1}{T} \widehat{G}(T)\to0$.
Thus, evaluating \cref{ES7.1C} at $v^\epsilon_*$, and applying It\^o's formula,
we obtain
\begin{equation*}
\limsup_{T\to\infty}\,\frac{1}{T}\,
\Exp^{v^\epsilon_*}_x\biggl[\int_0^{T}\rc^\epsilon_{v^\epsilon_*}
\bigl(X_s\bigr) \,\D{s}\biggr] \,\ge\, \Hat{\varrho}^\epsilon\,,
\end{equation*}
from which it follows that $\varrho^\epsilon_*\ge\Hat{\varrho}^\epsilon$.
This of course implies that $\varrho^\epsilon_*=\Hat{\varrho}^\epsilon$,
since \cref{ES7.1C} has no bounded from
below solutions for $\Hat{\varrho}^\epsilon<\varrho^\epsilon_*$.
\end{proof}

The following corollary concerns the construction 
of continuous precise $\epsilon$-optimal controls.
It follows directly from \cref{T7main} and the method
in \cite[Theorem~5.5]{AP16a}.

\begin{corollary}\label{C7.1}
For any given $\epsilon>0$ and $\Tilde{v}$ satisfying \cref{A2.2},
there exist
$R = R(\epsilon)>0$ and a continuous precise control 
$v_\epsilon \in \Ussm$ such that $v_\epsilon \equiv \tilde{v}$ on $\Bar{B}^c_R$,  
and
\begin{equation*}
\int_{\RR^d\times\Act}\rc(x,u)\,\uppi_{v_{\epsilon}}(\D{x},\D{u})\,\le\,
\varrho_* + \epsilon\,.
\end{equation*}
\end{corollary}

\section*{Acknowledgments}
This research was supported in part by 
the Army Research Office through grant W911NF-17-1-001, and
in part by the National Science Foundation through grants DMS-1715210,
CMMI-1538149 and DMS-1715875,
and in part by Office of Naval Research through grant N00014-16-1-2956
and was approved for public release under DCN \#43-5439-19.


\def\polhk#1{\setbox0=\hbox{#1}{\ooalign{\hidewidth
  \lower1.5ex\hbox{`}\hidewidth\crcr\unhbox0}}}


\begin{bibdiv}
\begin{biblist}

\bib{BenLi-84}{book}{
      author={Bensoussan, A.},
      author={Lions, J.-L.},
       title={Impulse control and quasivariational inequalities},
      series={$\mu $},
   publisher={Gauthier-Villars, Montrouge; Heyden \& Son, Inc., Philadelphia,
  PA},
        date={1984},
        ISBN={2-04-015577-5},
        note={Translated from the French by J. M. Cole},
      review={\MR{756234}},
}

\bib{Menaldi-99}{incollection}{
      author={Menaldi, Jose-Luis},
      author={Robin, Maurice},
       title={On optimal ergodic control of diffusions with jumps},
        date={1999},
   booktitle={Stochastic analysis, control, optimization and applications},
      series={Systems Control Found. Appl.},
   publisher={Birkh\"{a}user Boston, Boston, MA},
       pages={439\ndash 456},
      review={\MR{1702974}},
}

\bib{ACPZ19}{article}{
      author={Arapostathis, Ari},
      author={Caffarelli, Luis},
      author={Pang, Guodong},
      author={Zheng, Yi},
       title={Ergodic control of a class of jump diffusions with finite
  {L}\'{e}vy measures and rough kernels},
        date={2019},
     journal={SIAM J. Control Optim.},
      volume={57},
      number={2},
       pages={1516\ndash 1540},
      review={\MR{3942851}},
}

\bib{ABP15}{article}{
      author={Arapostathis, Ari},
      author={Biswas, Anup},
      author={Pang, Guodong},
       title={Ergodic control of multi-class {$M/M/N+M$} queues in the
  {H}alfin-{W}hitt regime},
        date={2015},
        ISSN={1050-5164},
     journal={Ann. Appl. Probab.},
      volume={25},
      number={6},
       pages={3511\ndash 3570},
      review={\MR{3404643}},
}

\bib{AP16a}{article}{
      author={Arapostathis, Ari},
      author={Pang, Guodong},
       title={Ergodic diffusion control of multiclass multi-pool networks in
  the {H}alfin-{W}hitt regime},
        date={2016},
     journal={Ann. Appl. Probab.},
      volume={26},
      number={5},
       pages={3110\ndash 3153},
      review={\MR{3563203}},
}

\bib{APS19}{article}{
      author={Arapostathis, Ari},
      author={Pang, Guodong},
      author={Sandri\'{c}, Nikola},
       title={Ergodicity of a {L}\'{e}vy-driven {SDE} arising from multiclass
  many-server queues},
        date={2019},
     journal={Ann. Appl. Probab.},
      volume={29},
      number={2},
       pages={1070\ndash 1126},
      review={\MR{3910024}},
}

\bib{ABG12}{book}{
      author={Arapostathis, A.},
      author={Borkar, V.~S.},
      author={Ghosh, M.~K.},
       title={Ergodic control of diffusion processes},
      series={Encyclopedia of Mathematics and its Applications},
   publisher={Cambridge University Press},
     address={Cambridge},
        date={2012},
      volume={143},
      review={\MR{2884272}},
}

\bib{PMB00}{article}{
      author={Dai~Pra, P.},
      author={Di~Masi, G.~B.},
      author={Trivellato, B.},
       title={Pathwise optimality in stochastic control},
        date={2000},
        ISSN={0363-0129},
     journal={SIAM J. Control Optim.},
      volume={39},
      number={5},
       pages={1540\ndash 1557},
         url={https://doi.org/10.1137/S0363012998334778},
      review={\MR{1825592}},
}

\bib{PWM04}{article}{
      author={Dai~Pra, P.},
      author={Runggaldier, W.~J.},
      author={Tolotti, M.},
       title={Pathwise optimality for benchmark tracking},
        date={2004},
        ISSN={0018-9286},
     journal={IEEE Trans. Automat. Control},
      volume={49},
      number={3},
       pages={386\ndash 395},
         url={https://doi.org/10.1109/TAC.2004.824467},
      review={\MR{2062251}},
}

\bib{AA17}{article}{
      author={Arapostathis, Ari},
       title={Some new results on sample path optimality in ergodic control of
  diffusions},
        date={2017},
     journal={IEEE Trans. Automat. Control},
      volume={62},
      number={10},
       pages={5351\ndash 5356},
      review={\MR{3708911}},
}

\bib{Ghosh-97}{article}{
      author={Ghosh, M.~K.},
      author={Arapostathis, A.},
      author={Marcus, S.~I.},
       title={Ergodic control of switching diffusions},
        date={1997},
        ISSN={0363-0129},
     journal={SIAM J. Control Optim.},
      volume={35},
      number={6},
       pages={1952\ndash 1988},
      review={\MR{1478649}},
}

\bib{AP18}{article}{
      author={Arapostathis, Ari},
      author={Pang, Guodong},
       title={Infinite-horizon average optimality of the {N}-network in the
  {H}alfin-{W}hitt regime},
        date={2018},
     journal={Math. Oper. Res.},
      volume={43},
      number={3},
       pages={838\ndash 866},
      review={\MR{3846075}},
}

\bib{AP19}{article}{
      author={Arapostathis, Ari},
      author={Pang, Guodong},
       title={Infinite horizon asymptotic average optimality for large-scale
  parallel server networks},
        date={2019},
     journal={Stochastic Process. Appl.},
      volume={129},
      number={1},
       pages={283\ndash 322},
      review={\MR{3906999}},
}

\bib{ADPZ19}{article}{
      author={Arapostathis, Ari},
      author={Das, Anirban},
      author={Pang, Guodong},
      author={Zheng, Yi},
       title={Optimal control of {M}arkov-modulated multiclass many-server
  queues},
        date={2019},
     journal={Stochastic Systems},
      volume={9},
      number={2},
       pages={155\ndash 181},
}

\bib{APZ19b}{article}{
      author={Arapostathis, A.},
      author={Pang, G.},
      author={Zheng, Y.},
       title={Optimal scheduling of critically loaded multiclass many-server
  queues with service interruptions},
        date={to appear},
     journal={ArXiv e-prints},
}

\bib{GS72}{book}{
      author={G\={i}hman, \u{I}.~\={I}.},
      author={Skorohod, A.~V.},
       title={Stochastic differential equations},
   publisher={Springer-Verlag, New York-Heidelberg},
        date={1972},
        note={Translated from the Russian by Kenneth Wickwire, Ergebnisse der
  Mathematik und ihrer Grenzgebiete, Band 72},
      review={\MR{0346904}},
}

\bib{Gyongy-96}{article}{
      author={Gy\"{o}ngy, Istv\'{a}n},
      author={Krylov, Nicolai},
       title={Existence of strong solutions for {I}t\^{o}'s stochastic
  equations via approximations},
        date={1996},
        ISSN={0178-8051},
     journal={Probab. Theory Related Fields},
      volume={105},
      number={2},
       pages={143\ndash 158},
         url={https://doi-org.ezproxy.lib.utexas.edu/10.1007/BF01203833},
      review={\MR{1392450}},
}

\bib{Skorokhod-89}{book}{
      author={Skorokhod, A.~V.},
       title={Asymptotic methods in the theory of stochastic differential
  equations},
      series={Translations of Mathematical Monographs},
   publisher={American Mathematical Society, Providence, RI},
        date={1989},
      volume={78},
        note={Translated from the Russian by H. H. McFaden},
      review={\MR{1020057}},
}

\bib{AHP19}{article}{
      author={Hmedi, Hassan},
      author={{Arapostathis}, Ari},
      author={{Pang}, Guodong},
       title={On uniform stability of certain parallel server networks with no
  abandonment in the {H}alfin--{W}hitt regime},
        date={2019},
     journal={ArXiv e-prints},
      volume={1907.04793},
      eprint={https://arxiv.org/abs/1907.04793},
}

\bib{AHP18}{article}{
      author={Arapostathis, A.},
      author={Hmedi, H.},
      author={Pang, G.},
       title={On uniform exponential ergodicity of {M}arkovian multiclass
  many-server queues in the {H}alfin-{W}hitt regime},
        date={2018},
     journal={ArXiv e-prints},
      volume={1812.03528v1},
      eprint={https://arxiv.org/abs/1812.03528v1},
}

\bib{AHPS19}{collection}{
      author={Arapostathis, Ari},
      author={Hmedi, Hassan},
      author={Pang, Guodong},
      author={Sandri\'{c}, Nikola},
      editor={Yin, George},
      editor={Zhang, Qing},
       title={Uniform polynomial rates of convergence for a class of
  {L}\'evy-driven controlled {SDE}s arising in multiclass many-server queues},
      series={{M}odeling, {S}tochastic {C}ontrol, {O}ptimization, and
  {A}pplications. The IMA Volumes in Mathematics and its Applications},
   publisher={Springer, Cham},
        date={2019},
      volume={164},
}

\bib{PE82}{article}{
      author={Echeverr\'{i}a, Pedro},
       title={A criterion for invariant measures of {M}arkov processes},
        date={1982},
        ISSN={0044-3719},
     journal={Z. Wahrsch. Verw. Gebiete},
      volume={61},
      number={1},
       pages={1\ndash 16},
         url={https://doi-org.ezproxy.lib.utexas.edu/10.1007/BF00537221},
      review={\MR{671239}},
}

\bib{GM02}{book}{
      author={Garroni, Maria~Giovanna},
      author={Menaldi, Jose~Luis},
       title={Second order elliptic integro-differential problems},
      series={Chapman \& Hall/CRC Research Notes in Mathematics},
   publisher={Chapman \& Hall/CRC, Boca Raton, FL},
        date={2002},
      volume={430},
      review={\MR{1911531}},
}

\bib{Li03}{incollection}{
      author={Li, C.~W.},
       title={Lyapunov exponents of nonlinear stochastic differential equations
  with jumps},
        date={2003},
   booktitle={Stochastic inequalities and applications},
      series={Progr. Probab.},
      volume={56},
   publisher={Birkh\"{a}user, Basel},
       pages={339\ndash 351},
      review={\MR{2073440}},
}

\bib{GilTru}{book}{
      author={Gilbarg, David},
      author={Trudinger, Neil~S.},
       title={Elliptic partial differential equations of second order},
     edition={Second},
      series={Grundlehren der Mathematischen Wissenschaften},
   publisher={Springer-Verlag, Berlin},
        date={1983},
      volume={224},
        ISBN={3-540-13025-X},
      review={\MR{737190}},
}

\bib{AA-Harnack}{article}{
      author={Arapostathis, Ari},
      author={Ghosh, Mrinal~K.},
      author={Marcus, Steven~I.},
       title={Harnack's inequality for cooperative weakly coupled elliptic
  systems},
        date={1999},
     journal={Comm. Partial Differential Equations},
      volume={24},
      number={9-10},
       pages={1555\ndash 1571},
      review={\MR{1708101}},
}

\bib{Ichihara-13}{article}{
      author={Ichihara, Naoyuki},
      author={Sheu, Shuenn-Jyi},
       title={Large time behavior of solutions of {H}amilton-{J}acobi-{B}ellman
  equations with quadratic nonlinearity in gradients},
        date={2013},
     journal={SIAM J. Math. Anal.},
      volume={45},
      number={1},
       pages={279\ndash 306},
      review={\MR{3032978}},
}

\end{biblist}
\end{bibdiv}
\end{document}